\documentclass[reqno]{amsart}

\usepackage{graphicx}
\usepackage{mathrsfs}
\usepackage{enumerate}
\usepackage{amsmath}
\usepackage{amsfonts}
\usepackage{latexsym}
\usepackage{wrapfig}
\usepackage{float}
\usepackage{verbatim}
\usepackage[hidelinks]{hyperref}
\usepackage{url}
\usepackage{bbold}
\usepackage{color}
\usepackage[dvipsnames]{xcolor}
\usepackage{subcaption}
\usepackage{amssymb}

\usepackage{pgf,tikz,pgfplots, tikz-cd}
\usetikzlibrary{arrows}
\pgfplotsset{compat=1.15}
\usepackage{mathrsfs}
\usetikzlibrary{arrows}
\usepackage[export]{adjustbox}

\definecolor{grey}{rgb}{0.8,0.8,0.8}

\newcommand{\A}{\mathscr{A}}

\newcommand{\F}{\mathcal{F}}
\newcommand{\g}{\gamma}
\newcommand{\M}{\mathbf{M}}

\renewcommand{\S}{\mathbf{S}}
\renewcommand{\t}{\tau}
\newcommand{\T}{\mathscr{T}}
\newcommand{\X}{\mathscr{X}}
\newcommand{\x}{\mathbf{x}}

\newcommand{\y}{\mathbf{y}}
\newcommand{\Z}{\mathbb{Z}}
\newcommand{\SM}{(\mathbf{S}, \mathbf{M})}

\newtheorem{theorem}{Theorem}[section]
\newtheorem*{theorem*}{Theorem}
\newtheorem{lemma}[theorem]{Lemma}
\newtheorem{proposition}[theorem]{Proposition}
\newtheorem{crl}[theorem]{Corollary}

\theoremstyle{definition}
\newtheorem{definition}[theorem]{Definition}
\newtheorem{example}[theorem]{Example}

\newtheorem*{conjecture*}{Conjecture}

\theoremstyle{remark}
\newtheorem{remark}[theorem]{Remark}

\numberwithin{equation}{section}

\newcommand{\dfn}[1]{\emph{#1}}

\newcommand{\todo}[1]{\textcolor{DarkOrchid}{\textit{#1}}}

\usepackage{pgf,tikz,pgfplots}
\pgfplotsset{compat=1.15}
\usepackage{mathrsfs}
\usetikzlibrary{arrows}

\begin{document}

\title{Quivers from non-orientable surfaces}

\author[V. Bazier-Matte]{V\'eronique Bazier-Matte}
\address{Department of Mathematics, University of Connecticut, 06269, United States of America}
\email{veronique.bazier-matte@uconn.edu}

\author[F. He]{Fenghuan He}
\address{Commonwealth School, Boston, MA}
\email{lhe@commschool.org}

\author[R. Huang]{Ruiyan Huang}
\address{Yale University, New Haven, CT}
\email{ruiyan.huang@yale.edu}

\author[H. Luo]{Hanyi Luo}
\address{Harvey Mudd College, Claremont, CA}
\email{hluo@hmc.edu}

\author[K. Wright]{Kayla Wright}
\address{University of Minnesota, Twin Cities, Minneapolis, MN}
\email{kaylaw@umn.edu}

\thanks{}  

\definecolor{grey}{rgb}{0.75,0.75,0.75}

\begin{abstract}
We associate a quiver to a quasi-triangulation of a non-orientable marked surface and define a notion of quiver mutation that is compatible with quasi-cluster algebra mutation defined by Dupont and Palesi \cite{DP15}. Moreover, we use our quiver to show the unistructurality of the quasi-cluster algebra arising from the M\"obius strip.
\end{abstract}


\maketitle

\tableofcontents


\section{Introduction}
Cluster algebras were originally defined by Fomin and Zelevinsky in \cite{FZ02} in order to give a combinatorial characterization of total positivity and canonical bases in reductive groups, but have since been found to be helpful in many other areas of mathematics such as algebraic combinatorics, symplectic geometry, Teichm\"uller theory, and various areas of physics. 

Since the birth of cluster algebras, there have been various generalizations. In our paper, we will be exploring a generalization of cluster algebras from surfaces known as \textit{quasi-cluster algebras}. These algebras were defined via a quasi-triangulation of a non-orientable surface by Dupont and Palesi \cite{DP15}, extending the surface model for cluster algebras developed by \cite{FST08} for a triangulation of an orientable surface. There have been other generalizations besides the notion of quasi-cluster algebras that are active areas of research. 

For example, generalized cluster algebras were introduced in \cite{CS2014} by Chekhov and Shapiro, as another generalization of cluster algebras from surfaces where we allow ``orbifold points" of arbitrary integer orders. In these generalized cluster algebras, the binomial exchange relation for cluster variables of cluster algebras is replaced by polynomial exchange relation to correctly model the behavior of the surface. Lam and  Pylyavskyy also generalized the binomial exchange relation for cluster algebras in their definition of Laurent Phenomenon algebras, also known as LP algebras, where they allow exchange relations to be arbitrary irreducible polynomials, rather than binomials \cite{lp}. The relaxation of exchange relations led to investigation of what types of algebras held the key features of classical cluster algebras -- for example, Lam and Pylyavskyy wanted to get at a larger class of algebras that satisfy the Laurent Phenomenon, moreover with positive coefficients.

For quasi-cluster algebras specifically, many of the classical cluster algebraic results first posed by Fomin and Zelevinsky such as Laurent Phenomenon and Finite Type Classification were proven by Dupont and Palesi \cite{DP15}. Since this original paper by Dupont and Palesi, there have been other works on quasi-cluster algebras. For instance, connecting two generalizations of cluster algebras, Jon Wilson showed that quasi-cluster algebras arising from unpunctured non-orientable surfaces also have an associated LP algebra \cite{LP16}. Wilson continued working on quasi-cluster algebras and investigated topological questions about the associated cluster complex. Namely, he proved that the quasi-cluster complex of the M\"obius strip is shellable and homeomorphic to a PL sphere \cite{wilson2015shellability}. In 2019, Wilson also proved positivity of the coefficients in the Laurent expansions for quasi-cluster variables \cite{Wil19}. In particular, positivity was proven by extending the cluster expansion formula using snake graphs proven in the orientable case in \cite{MSW13}. In addition to these results from Wilson, the first, third and fourth author investigated combinatorial problems related to quasi-cluster algebras. In particular, they developed a formula for the number of triangulations of a M\"obius strip, depending on its number of marked points \cite{BHL20}.

In this paper, we continue the exploration of quasi-cluster algebras by defining a quiver associated to a quasi-triangulation of a non-orientable surface. This quiver allows us to perform mutations without the data of the triangulation of the surface and is compatible with the original definition of mutation given in \cite{DP15} which is our main result stated in Theorem \ref{theorem::compatibilityofmutation}.

Our paper is organized as follows: in Section \ref{section::prelims}, we review cluster algebras from orientable surfaces and quasi-cluster algebras from non-orientable surfaces. From here, our main works begins in Section \ref{section::ourquiver} where we define a quiver associated to a non-orientable surface. In this section, we define our notion of mutation that is compatible for quasi-cluster algebra mutation defined by \cite{DP15} as well as prove some desirable properties of our quiver mutation. We also justify why our quiver is a useful object to consider and make a connection to work of Jon Wilson for quasi-cluster algebras used in \cite{Wil18}, \cite{Wil19}, \cite{Wil20} and \cite{wilson2015shellability}. Lastly, as a bonus, in our final Section \ref{section::unistructurality}, we prove that the quasi-cluster algebra arising from a M\"obius strip is unistructural.

{\bf Acknowledgements:} The authors would like to thank Canada/USA Mathcamp for bringing us all together. We would also like to thank Gregg Musiker for helping conversations and insights into the writing of this paper. 
 
\section{Preliminaries} \label{section::prelims}
In this section, we review the notion of a marked surface, assuming nothing about the orientability of the surface. We then restrict our viewpoint to orientable surfaces and recall the definition of a cluster algebra from an orientable surface defined by Fomin, Shapiro and Thurston \cite{FST08}. Lastly, we will address the non-orientable case by recalling Dupont and Palesi's work in defining a quasi-cluster algebras arising from a non-orientable surfaces \cite{DP15}.

\subsection{Marked Surfaces}
\begin{definition} \label{definition::markedsurface}
Let $\S$ be a compact, connected Riemann surface with boundary $\partial \S$. Let $\M$ be a finite set of points, we call \textit{marked points}, contained in $\S$ such that each connected component of $\partial \S$ has at least one point of $\M$. We say the pair $\SM$ is a \textit{marked surface}. 
\end{definition}


\begin{remark}
For simplicity, we will assume that all marked points are contained in $\partial \S$ i.e. our marked surfaces have no interior marked points known as punctures. In addition, we wish to exclude the cases of pairs $\SM$ that do not admit a triangulation. To that end, we disallow $\SM$ to be a monogon, bigon or a triangle. 
\end{remark}

\begin{definition}\label{definition::regulararc}
A regular \textit{arc} $\t$ in a marked surface $\SM$ is a curve in $\S$, considered up to isotopy relative its endpoints \footnote{Let $X,Y$ be two topological spaces and let $f,g:X \to Y$ be two continuous maps. A homotopy between $f,g$ is another continuous map $h: X \times [0,1] \to Y$ such that $h(x,0) = f(x)$ and $h(x,1) = g(x)$ for all $x \in X$. An isotopy is a homotopy $h$ that induces a homeomorphism from $X$ onto $h(X,t)$ for all $0 \leq t \leq 1$ i.e. for any $0 \leq t \leq 1$, the map $h(\cdot,t): X \to h(X,t)$ is a homeomorphism. } such that
\begin{enumerate}
    \item the endpoints of $\t$ are $\M$,
    \item $\t$ has no self-intersections itself, except at its endpoints,
    \item except for the endpoints, $\t \cap \M = \emptyset$ and $\t \cap \partial \S = \emptyset$;
    \item and $\t$ does not cut out a monogon or bigon. (In other words, $\t$ is not contractible into $\M$ or onto the boundary of $\S$.)
\end{enumerate}
We call arcs that cut out a bigon \textit{boundary arcs} that we consider as part of the data of $\SM.$ We denote the set of boundary arcs by $\mathbf{B}$.
\end{definition}

We say that two (regular) arcs $\t, \t'$ in $\SM$ are \textit{compatible} if up to isotopy, they do not intersect one another. More formally, define $\text{cross}(\t, \t')$ to be the minimum of the number of crossings  of $\alpha$ and $\alpha$' where $\alpha$ is an arc isotopic to $\t$ and $\alpha'$ is an arc isotopic to $\t'$
%

We say regular arcs $\t, \t'$ are \textit{compatible} if cross$(\t, \t') = 0$. A maximal collection of pairwise compatible arcs in $\SM$ is called a \textit{triangulation of $\SM$}.

\begin{example}
Figure \ref{Fig::Triangulation} shows a triangulation of a marked surface with three boundary components; each of the boundary components contains one marked point.
 \begin{figure}[h]
 \centering
 \begin{tikzpicture} [scale = 0.6]
    \draw[very thick] (0,0) circle (3cm);
    \draw[very thick] (-1.25,0) circle (0.75cm);
    \draw[very thick] (1.25,0) circle (0.75cm);
    
    \tikzstyle{p}=[circle,fill, scale=0.3]
    \node[p] (1) at (-0.5, 0) {};
    \node[p] (2) at (0.5, 0) {};
    \node[p] (3) at (0, -3) {};
    
    \draw (1) -- (2);
    \draw (1) -- (3);
    \draw (2) -- (3);
    \draw (-0.5,0) arc (0:180:1cm);
    \draw (-2.5,0) .. controls (-2.5,-0.5) and (-2,-2) .. (3);
    \draw (1) .. controls (-0.5,0.25) and (0.75,1) .. (1.25,1);
    \draw (1.25,-1) arc(-90:90:1);
    \draw (2) .. controls (0.5,-0.25) and (0.75,-1) .. (1.25,-1);
    \draw (2.75,0) arc (0:180:1.625);
    \draw (2.75,0) .. controls (2.75,-0.75) and (2,-2) .. (3);
 \end{tikzpicture}
 \caption{Triangulation of an orientable marked surface}
 \label{Fig::Triangulation}
 \end{figure}
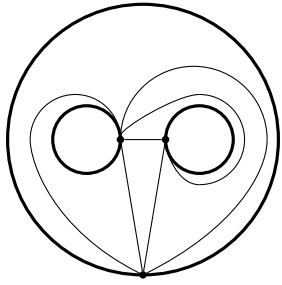
\end{example}

In \cite{FST08}, the authors insist that $\S$ is orientable in order to define a cluster algebra from a triangulation of $\SM$. However, in \cite{DP15}, they extend the notion of a cluster algebra in the case of non-orientable surfaces. We first focus on the orientable case first studied to create a (classical) cluster algebra.

\begin{subsection}{Cluster Algebras from Orientable Surfaces}\label{subsection::orientablecase}
Let $\SM$ be a marked surface such that $\S$ is orientable in this subsection. Let $T = \{\tau_1, \tau_2, \dots, \tau_n\}$ be a triangulation of $\SM$. 
We associate an indeterminate $y_b$ to each boundary segment $b$ of $\partial S$, such that $\y = \left\{ y_b | b \in \mathbb{B} \right\}$ is algebraically independent. The elements of $\y$ are called \emph{geometric coefficients}.
Let $k \in \mathbb{N}$ and suppose $\F$ is a field of rational functions with $n+k$ indeterminates, that is
\[ \F = \left\{ \left. \frac{p(x_1, x_2 \dots, x_{n}, y_1, y_2 \dots y_m)}{q(x_1, x_2 \dots, x_{n}, y_1, y_2, \dots, y_m)} \right| p, q \in \Z[x_1, x_2 \dots, x_{n}, y_1, y_2 \dots, y_m]\right\}. \]
 We set the \dfn{ground ring} to be
\[ \mathbb{Z}\mathbb{P} = \mathbb{Z}[y_b^{\pm 1}\mid b\in \SM].\]

\begin{definition}
 An \emph{initial seed} is a triplet $(\x,\y,T)$ such that: 
 \begin{itemize}
     \item $T$ is a triangulation of $\SM$;
     \item $\x = \{ x_{\t} | \t \in T \}$ is a \emph{cluster} and together with $\y$, it forms a generating set of $\F$.
 \end{itemize}
\end{definition}

In other words, an initial seed is a triangulation with each of its arcs associated with a variable $x_i$ if it is internal and $y_i$ if it is on the boundary. With this, we are able to obtain new seeds from initial seed by a recursive process known as \textit{mutation}. The seeds generated from mutation form \textit{cluster variables}.

\begin{definition}\label{definition::orientablemutation}
 Let $(\x,\y, T)$ be an initial seed and $t \in T$. The \textbf{mutation} of $(\x,\y, T)$ in the direction $t$ transforms the seed $(\x, \y, T)$ in a new \textit{seed} $\mu_t(\x, \y, T) = (\mu_t(\x), \y, \mu_t(T))$ where $\x = \x \setminus \{ x_t \} \cup \{ x_{t'} \}$ with $x_{t'}$ such that if the arc $t$ is the diagonal in a quadrilateral $abcd$ that creates two triangles as shown in Figure \ref{Fig::Mutation}, then the relation is the following: 
    \begin{equation} \label{Eq::Mutation1}
     x_tx_t'=x_ix_k+x_jx_l.
    \end{equation}
This is called the \emph{Ptolemy relation} because it stems from the relation that holds among the lengths of the sides and diagonals of a regular quadrilateral with unit length sides inscribed in a circle known as Ptolemy's Theorem.
    
 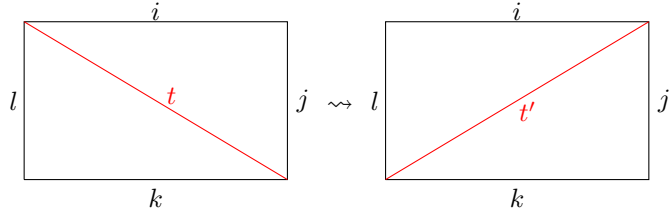
\begin{figure}[h]
 \centering
  \begin{tikzpicture} [baseline={([yshift=-.5ex]current bounding box.center)}, scale=0.7]
\draw (0,0)  -| (5,3) 
    node[pos=0.25,below] {$k$} 
    node[pos=0.75,right] {$j$}
    -| (0,0);
    \node at (2.5,3.2) {$i$};
    \node at (-0.2,1.5) {$l$};
    
    \draw [red](5,0) --  (0,3);
    \node [red] at (2.8,1.6){$t$};
   
 \end{tikzpicture}
   $\rightsquigarrow$
  \begin{tikzpicture} [baseline={([yshift=-.5ex]current bounding box.center)}, scale=0.7]
 \draw (0,0)  -| (5,3) 
    node[pos=0.25,below] {$k$} 
    node[pos=0.75,right] {$j$}
    -| (0,0);
    \node at (2.5,3.2) {$i$};
    \node at (-0.2,1.5) {$l$};
    
    \draw [red](0,0) --  (5,3);
    \node [red] at (2.7,1.3){$t'$};
   
 \end{tikzpicture}
 \caption{Mutation via flipping the diagonal of a quadrilateral.}
 \label{Fig::Mutation}
 \end{figure}
 

\end{definition}

Let $\T$ be the union of all triangulations of a marked surfaces $\SM$ obtainable by a sequence of mutations starting from a fixed initial seed $(\x, \y, T)$. The cluster algebra $\A\SM$ is the algebra generated by $\T$ over the ground ring $\mathbb{Z}\mathbb{P}$ i.e. $\A\SM = \mathbb{Z}\mathbb{P}[\T]$. 
With this model using a surface, we have the following correspondence given by \cite{FST08}. 

\begin{theorem}{\cite{FST08}}
Let $\SM$ be a marked surface and let $\A = \A\SM$ be the associated cluster algebra. There are bijections
$$\{\text{cluster variables of }\A\} \longleftrightarrow \{\text{arcs of } \SM\}$$
$$\{\text{clusters of }\A\} \longleftrightarrow \{\text{triangulations of } \SM\},$$
where quasi-cluster variables are the generators of the associated quasi-cluster algebra, grouped into sets called clusters. 

Moreover, let $T$ be a triangulation of $\SM$, let $\t \in T$ be an internal arc and let $\t'$ be the arc obtained by flipping $\t$ in $T$. Then, cluster mutation $\mu_t(\x_T)$ in direction $t$ is compatible with flip $\mu_t(T)$ in a quadrilateral of the arc $\t$, that is,

$$\mu_t(\x_T) = (\x_T \setminus \{x_{\t}\}) \cup \{x_{\t'}\} $$

\noindent corresponds to 

$$\mu_t(T) = (T \setminus \{\t_t\}) \cup \{\t_{t'}\}.$$
\end{theorem}

\end{subsection}

\begin{subsection}{Quasi-Cluster Algebras}\label{subsection::quasiCA}
We now extend the notion of a cluster algebra to a non-orientable surface following the work of \cite{DP15}. To properly describe this constructions, we review some topology of non-orientable surfaces.

\begin{subsubsection}{Non-Orientable Surfaces}\label{subsubsection::topology}
All compact, connected surfaces have been classified. Namely, any non-orientable surface is homeomorphic to a connected sum of $k$ projective planes $\mathbb{R}\mathbb{P}^2$, where the number $k$ is called the non-orientable genus of the surface.

To visually represent non-orientable surfaces, we use an identification of the real projective plane $\mathbb{R}\mathbb{P}^2$ as the quotient of the 2-sphere under the antipodal map. Cutting this sphere along the equator, the projective plane is homeomorphic to a disk with opposite points on the boundary identified, which is called a \textit{crosscap} which we denote of the surface using $\bigotimes$. Using the idea of cross-caps, any genus $k$ non-orientable surface can be represented as a sphere with $k$ open disks removed and the opposite points of each boundary components identified.

\begin{example}
A disk with one crosscap in the interior is homeomorphic to a M\"obius strip. We demonstrate this homeomorphism in Figure \ref{Fig::homeomorphism}.

\begin{figure}
 \centering
 \begin{tikzpicture}[baseline={([yshift=-.5ex]current bounding box.center)}, scale=0.5]
    \draw[very thick] (0,0) circle (2cm);
  
    \draw[thick] (0, 0) circle (0.25cm);
    \draw[rotate=45] (0,-0.25) -- (0,0.25);
    \draw[rotate=45] (-0.25,0) -- (0.25, 0);
 \end{tikzpicture}
 $\cong$
  \begin{tikzpicture}[baseline={([yshift=-.5ex]current bounding box.center)}, scale=0.5]
    \draw[very thick] (0,0) circle (2cm);
  
    
    \filldraw[color=black, fill = gray] (0, 0) circle (0.25cm);
    \draw[latex-] (0.3,0.5)..  controls (0,0.6).. (-0.3,0.5);
    \draw[latex-] (-0.3,-0.5)..  controls (0,-0.6).. (0.3,-0.5);
    
    \draw[dotted,red] (-2,0)--(-0.25,0);
    \draw[dotted,blue] (2,0)--(0.25,0);
    
 \end{tikzpicture}
  $\cong$
  \begin{tikzpicture}[baseline={([yshift=-.5ex]current bounding box.center)}, scale=0.5]
  \draw (-2,4)--(0,4);
  \draw[latex-] (-1,2.5)--(-2,2.5);
  \draw (-1,2.5)--(0,2.5);
  \draw (-1,2)--(-2,2);
  \draw[-latex] (0,2)--(-1,2);
  \draw (-2,0.5)--(0,0.5);
  \draw[blue] (-2,4) -- (-2,3.5);
  \draw[blue,latex-] (-2,3.5) -- (-2,2.5);
  \draw[blue,latex-] (-2,1) -- (-2,2);
  \draw[blue] (-2,1) -- (-2,0.5);
  \draw[red,latex-] (0,3) -- (0,4);
  \draw[red] (0,3) -- (0,2.5);
  \draw[red] (0,1.5) -- (0,2);
  \draw[red,-latex] (0,0.5) -- (0,1.5);
 \end{tikzpicture}
 $\cong$
   \begin{tikzpicture}[baseline={([yshift=-.5ex]current bounding box.center)}, scale=0.5]
  \draw (-2,4)--(0,4);
  \draw[latex-] (-1,2.5)--(-2,2.5);
  \draw (-1,2.5)--(0,2.5);
  \draw (-1,2)--(-2,2);
  \draw[-latex] (0,2)--(-1,2);
  \draw (-2,0.5)--(0,0.5);
  \draw[blue] (-2,4) -- (-2,3.5);
  \draw[blue,latex-] (-2,3.5) -- (-2,2.5);
  \draw[blue,latex-] (0,1) -- (0,2);
  \draw[blue] (0,1) -- (0,0.5);
  \draw[red,latex-] (0,3) -- (0,4);
  \draw[red] (0,3) -- (0,2.5);
  \draw[red] (-2,1.5) -- (-2,2);
  \draw[red,-latex] (-2,0.5) -- (-2,1.5);
 \end{tikzpicture}
 $\cong$
   \begin{tikzpicture}[baseline={([yshift=-.5ex]current bounding box.center)}, scale=0.5]
  \draw (-2,3.5)--(0,3.5);
  \draw (-2,2)--(0,2);
  \draw (-2,0.5)--(0,0.5);
  \draw[blue] (-2,3.5) -- (-2,3);
  \draw[blue,-latex] (-2,2) -- (-2,3);
  \draw[blue,latex-] (0,1) -- (0,2);
  \draw[blue] (0,1) -- (0,0.5);
  \draw[red,latex-] (0,2.5) -- (0,3.5);
  \draw[red] (0,2) -- (0,2.5);
  \draw[red] (-2,1.5) -- (-2,2);
  \draw[red,-latex] (-2,0.5) -- (-2,1.5);
 \end{tikzpicture}

 \caption{The disk with a single crosscap in the interior is the M\"obius strip.}
 \label{Fig::homeomorphism}
 \end{figure}
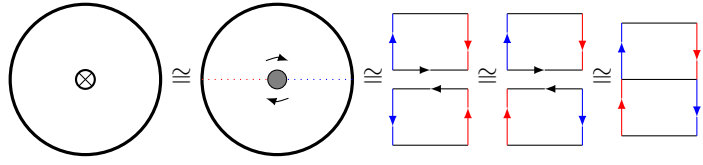

 \end{example}
One fact about non-orientable surfaces that we will employ later on is a formula for the Euler characteristic.
\begin{proposition}\label{prop::eulerchar}
The Euler characteristic of a non-orientable surface $\S$ of genus $k$ is given by $\chi(\S)=2-k$.
\end{proposition}
\end{subsubsection}

We now define the necessary data on our orientable surface that will eventually produce a \textit{quasi-cluster algebra}. In particular, we will be considering the same regular arcs from orientable case together with another special type of arc to non-orientable surfaces. From now on, assume that $\S$ need not be orientable. 

\begin{definition}\label{def::quasiarc}
A \textit{quasi-arc} in $\SM$ is a simple closed curve in the interior of $\S$ that traverses through a crosscap. To illustrate this, see Figure \ref{Fig::quasi_arc}.
\end{definition}

\begin{figure}
 \centering
 \begin{tikzpicture} [scale = 0.7]
    \draw[very thick] (0,0) circle (2cm);
  
    \draw[thick] (0, 0) circle (0.25cm);
    \draw[rotate=45] (0,-0.25) -- (0,0.25);
    \draw[rotate=45] (-0.25,0) -- (0.25, 0);

    \draw (0.25,0) ..  controls (1,0) and (0.55,0.95) ..  (0,1);
    \draw (0,1) .. controls (-0.55,0.95) and (-1,0)  .. (-0.25,0);
 \end{tikzpicture}
\caption{The quasi arc on a M\"obius strip.}
\label{Fig::quasi_arc}
\end{figure}
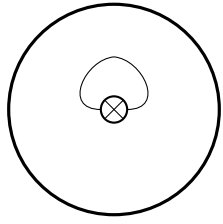

\begin{remark}
Note that in \cite{DP15}, the authors refer to quasi-arcs as either regular arcs or this quasi-arc we define in Definition \ref{def::quasiarc}. We choose to refer to ``arcs" as the union of regular arcs and quasi-arcs. We do this to reserve the language of quasi-arcs to the case we've defined in Definition \ref{def::quasiarc}.
\end{remark}

Using both regular arcs and quasi-arcs, we extend the notion of triangulation - allowing quasi-arcs. As before, we will assume two arcs are compatible if they do not cross each other up to isotopy.

\begin{definition}\label{def::quasitriangulation}
A \textit{quasi-triangulation} of $\SM$ is a maximal collection of compatible arcs (possibly using quasi-arcs). A quasi-triangulation is called a triangulation if it consists only of regular arcs.
\end{definition}

In quasi-triangulations, many different triangles can arise. 

\begin{definition}
A \emph{regular triangle} is a triangle cut by three regular arcs (some potentially on the boundary). An \emph{anti-self-folded} triangle is a triangle of a quasi-triangulation with two edges identified by an orientation-reversing isometry. The last type of triangle is a \emph{quasi-triangle} which is a triangle consisting of a quasi-arc inscribed in a loop based a marked point. See Figure \ref{Fig::typesoftriangles} for examples of these various types of triangles. 
\end{definition}

\begin{figure}
    \centering
    \begin{tikzpicture} [baseline={([yshift=-.5ex]current bounding box.center)}, scale=1]
    \draw[very thick] (0,0) circle (3cm)
    node [above] at (0,3){$i$}
    node [below] at (0,-3){$j$};
    
    \tikzstyle{p}=[circle,fill, scale=0.3]
    \node[p] (1) at (-3,0) {};
    \node[p] (2) at (3,0) {};
    
    \draw [thick](1,1) circle (0.25cm);
    \draw [thick](-1,-1) circle (0.25cm);
    \draw[xshift=-1cm,yshift=-1cm,rotate=45] (0,-0.25) -- (0,0.25);
    \draw[xshift=-1cm,yshift=-1cm,rotate=45] (-0.25,0) -- (0.25, 0);
    \draw[xshift=1cm,yshift=1cm,rotate=45] (0,-0.25) -- (0,0.25);
    \draw[xshift=1cm,yshift=1cm,rotate=45] (-0.25,0) -- (0.25, 0);
    
     \draw [xshift=-1cm,yshift=-1cm](0.25,0) ..  controls (1,0) and (0.55,-0.65) ..  (0,-0.67);
    \draw [xshift=-1cm,yshift=-1cm] (0,-0.67) .. controls (-0.55,-0.65) and (-1,0)  .. node[very near start, above]{$\t_1$} (-0.25,0);
    
    \draw (1) .. controls (-1,-0.25) and (-0.3,-0.2) .. node [very near end,right]{$\t_2$}(-0.2,-1.5);
    \draw (-0.2,-1.5) .. controls (-0.3,-2.5) and (-1,-2.75) .. (1);
    
    \draw (1) .. controls (-1,1) .. node [very near end,right]{$\t_3$}(0,0);
    \draw (0,0) .. controls (1,-1) .. (2);
    
    \draw (2) .. controls (1,0.25) and (0.3,0.2) .. (0.2,1.5);
    \draw (0.2,1.5) .. controls (0.3,2.5) and (1,2.75) .. node [near start,left]{$\t_4$}(2);
    
    \draw (1,0.75) .. controls (1.25,0.35) .. node [midway,above]{$\t_5$}(2);
    \draw (1,1.25) .. controls (1.27,1.5) .. (2);
 \end{tikzpicture}
\caption{An example of a quasi-triangulation on the M\"obius strip with two marked points. The triangle enclosed by arcs $i, \t_3, \t_4$ is an example of a regular triangle. The triangle enclosed by the arcs $\t_4, \t_5$ is an anti-self-folded triangle. And the triangle enclosed by $\t_2$ and the quasi-arc $\t_1$ is a quasi-triangle. }
    \label{Fig::typesoftriangles}
\end{figure}
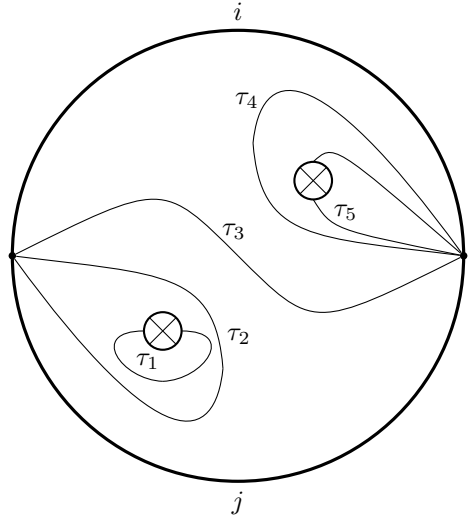

\begin{lemma}[\cite{DP15} Proposition 3.2]
\label{lem::typearc}
Let $T$ be a triangulation of $(\S,\M)$. An arc of $T$ cuts either \begin{enumerate} [i)]
    \item \label{type1} two different regular triangles;
    \item \label{type2} only one regular triangle;
    \item \label{type3} one regular triangle and one quasi-triangle;
    \item \label{type4} only one quasi-triangle.
\end{enumerate}
\end{lemma}

Following this lemma, we define \dfn{arcs of types \eqref{type1}, \eqref{type2}, \eqref{type3}} and \dfn{\eqref{type4}} according to their above description. Note that arcs of type \eqref{type1}, \eqref{type2} and \eqref{type3} are always regular arcs while arcs of type \eqref{type4} are quasi-arcs.

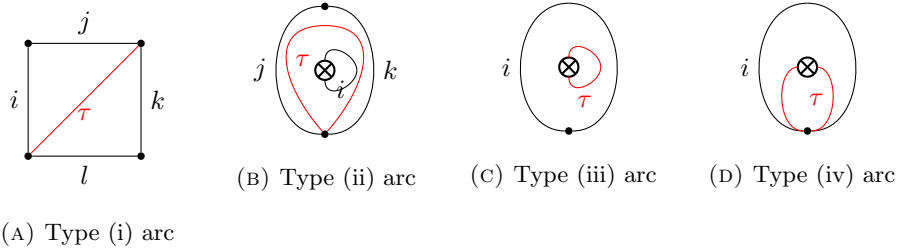
\begin{figure}
    \begin{subfigure}[t]{0.24\textwidth}
        \centering
            \[\begin{tikzpicture}[scale = 1.5]
             \tikzstyle{p}=[circle,fill, scale=0.3]
            \node[p] (1) at (0, 0) {};
            \node[p] (2) at (0, 1) {};
            \node[p] (3) at (1, 1) {};
            \node[p] (4) at (1, 0) {};
            \draw (1) -- node [midway, left] {$i$} (2);
            \draw (2) -- node [midway, above] {$j$}(3);
            \draw (3) -- node [midway, right] {$k$}(4);
            \draw (4) -- node [midway, below] {$l$}(1);
            \draw[red] (1) -- node [right, below] {$\t$} (3);
            \end{tikzpicture}\]
        \caption{Type \eqref{type1} arc}
     \end{subfigure}
     \hfill
     \begin{subfigure}[t]{0.24\textwidth}
        \centering
            \[ \begin{tikzpicture}[baseline={([yshift=-.5ex]current bounding box.center)}, scale=0.85]
            \tikzstyle{p}=[circle,fill, scale=0.3]
            \node[p] (1) at (0, -1) {};
            \node[p] (2) at (0, 1) {};
            \draw[thick] (0, 0) circle (0.15cm);
            \draw[rotate=45, thick] (0,-0.15) -- (0,0.15);
            \draw[rotate=45, thick] (-0.15,0) -- (0.15, 0);
            \draw (1) .. controls (-1,-1) and (-1,1) .. node [left] {$j$} (2);
            \draw (1) .. controls (1,-1) and (1,1) .. node [right] {$k$}(2);
            \draw (0,-0.15) .. controls (0,-0.5) and (0.5,-0.25) ..node [midway] {$i$} (0.5,0);
            \draw (0,0.15) .. controls (0,0.5) and (0.5,0.25) .. (0.5,0);
            
            \draw[red] (1) .. controls (-0.7,-0.2) and (-0.9,0.7) ..node [midway, right] {$\t$} (0,0.7);
            \draw[red] (1) .. controls (0.7,-0.2) and (0.9,0.7) .. (0,0.7);
            \end{tikzpicture}\]
            \caption{Type \eqref{type2} arc}
            \hfill
           \end{subfigure}
\begin{subfigure}[t]{0.24\textwidth}
            \[\begin{tikzpicture}[scale=0.85]
            \tikzstyle{p}=[circle,fill, scale=0.3]
            \node[p] (1) at (0, -1) {};
            \draw[thick] (0, 0) circle (0.15cm);
            \draw[rotate=45, thick] (0,-0.15) -- (0,0.15);
            \draw[rotate=45, thick] (-0.15,0) -- (0.15, 0);
            \draw (1) .. controls (-1,-1) and (-1,1) .. node [left]{$i$}(0,1);
            \draw (1) .. controls (1,-1) and (1,1) .. (0,1);
            \draw[red] (0,-0.15) .. controls (0,-0.5) and (0.5,-0.25) .. node [below]{$\t$}(0.5,0);
            \draw[red] (0,0.15) .. controls (0,0.5) and (0.5,0.25) .. (0.5,0);
            \end{tikzpicture}\]
        \caption{Type \eqref{type3} arc}
     \end{subfigure}
     \hfill
     \begin{subfigure}[t]{0.24\textwidth}
        \centering
            \[ \begin{tikzpicture}[baseline={([yshift=-.5ex]current bounding box.center)}, scale=0.85]
            \tikzstyle{p}=[circle,fill, scale=0.3]
            \node[p] (1) at (0, -1) {};
            \draw[thick] (0, 0) circle (0.15cm);
            \draw[rotate=45, thick] (0,-0.15) -- (0,0.15);
            \draw[rotate=45, thick] (-0.15,0) -- (0.15, 0);
            \draw (1) .. controls (-1,-1) and (-1,1) ..  node [midway, left]{$i$}(0,1);
            \draw (1) .. controls (1,-1) and (1,1) ..  (0,1);
            \draw[red] (1) .. controls (-0.5,-1) and (-0.5,0) .. (-0.15,0);
            \draw[red] (1) .. controls (0.5,-1) and (0.5,0) .. node [midway,left]{$\t$}(0.15,0);
            \end{tikzpicture}\]
        \caption{Type \eqref{type4} arc}
     \end{subfigure}
     \hfill
\caption{Four local configurations that can appear in any quasi-triangulation.}  
    \label{Fig::surfacelocalconfigs}
\end{figure}

\begin{example}
In Figure \ref{Fig::typesoftriangles}, an arc that cuts two different regular triangles i.e. is of type \eqref{type1}, is $\tau_3$. An arc that only cuts one different regular triangle i.e. is of type \eqref{type2}, is $\tau_5$. An arc that cuts one regular triangle and one quasi-triangle i.e. is of type \eqref{type3}, is $\tau_2$.
\end{example}

 From this notion of quasi-triangulation, we can extend the definition of a cluster algebra to this non-orientable setting. Namely, we can define an initial (quasi-) seed. 

\begin{definition} \label{def::quasiseed}
 An \textit{initial seed} is a triplet $(\x,\y,T)$ such that: \begin{itemize}
     \item $T$ is a (quasi-)triangulation of $\SM$;
     \item $\x = \{ x_\t | \t \in T \}$ is a \textit{cluster} and together with $\y$, it forms a generating set of $\F$.
 \end{itemize}
\end{definition}

From this, we define a notion of mutation of these (quasi-)seeds. 

\begin{proposition}\label{prop:mutation}
Let $\SM$ be a marked surface and let $T$ be a quasi-triangulation of $\SM$. Then for any arc $t \in T$, there exists a unique arc $t'\in \SM$ such that $t'\not= t$ and such that $\mu_t(T)=T \backslash \{t\} \sqcup  \{t'\}$ is a quasi-triangulation of $\SM$.
\end{proposition}

Replacing $t$ by $t'$ is also known as the \textit{quasi-mutation}. With this quasi-mutation, we also have quasi-exchange relations given in \cite{DP15}. Since there are different types of arcs in this setting, quasi-mutation and the quasi-exchange relations are defined differently for different types of arcs formalized in the following definition:

\begin{definition} \label{def::quasimutation}
Let $(\x, \y, T)$ be an initial (quasi-)seed and $t \in T$. The \textbf{mutation} of $(\x, \y, T)$ in the direction $t$ transforms the (quasi-)seed $(\x, \y, T)$ in a new (quasi-)seed $\mu_t(\x, \y, T) = (\mu_t(\x), \y, \mu_t(T))$ where $\x = \x \setminus \{ x_t \} \cup \{ x_{t'} \}$ with $x_{t'}$ such that:
\begin{enumerate}
    \item If the arc $t$ is the diagonal in a quadrilateral $ijkl$ as in Definition \ref{definition::orientablemutation}, the exchange relation is given by 
    \begin{equation} \label{Eq::Mutation1}
     x_tx_t'=x_ix_k+x_jx_l.
    \end{equation}
    
 \begin{figure}[h]
 \centering
  \begin{tikzpicture} [baseline={([yshift=-.5ex]current bounding box.center)}, scale=0.7]
\draw (0,0)  -| (5,3) 
    node[pos=0.25,below] {$k$} 
    node[pos=0.75,right] {$j$}
    -| (0,0);
    \node at (2.5,3.2) {$i$};
    \node at (-0.2,1.5) {$l$};
    
    \draw [red](5,0) --  (0,3);
    \node [red] at (2.8,1.6){$t$};
   
 \end{tikzpicture}
   $\rightsquigarrow$
  \begin{tikzpicture} [baseline={([yshift=-.5ex]current bounding box.center)}, scale=0.7]
 \draw (0,0)  -| (5,3) 
    node[pos=0.25,below] {$k$} 
    node[pos=0.75,right] {$j$}
    -| (0,0);
    \node at (2.5,3.2) {$i$};
    \node at (-0.2,1.5) {$l$};
    
    \draw [red](0,0) --  (5,3);
    \node [red] at (2.7,1.3){$t'$};
   
 \end{tikzpicture}
 \caption{Type i mutation}
 \label{Fig::Mutation1}
 \end{figure}
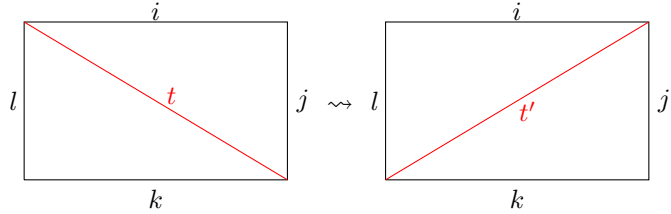

\item Consider an anti-self-folded-triangle, as in Figure \ref{Fig::Mutation2}. Let $t$ be the arc corresponding to the two identified edges. Then, the relation is 

    \begin{equation} \label{Eq::Mutation2}
   x_tx_t' = x_i.
    \end{equation}
    
    
\begin{figure}[h]
 \centering
 \begin{tikzpicture}[baseline={([yshift=-.5ex]current bounding box.center)}, scale=1.5]
            \tikzstyle{p}=[circle,fill, scale=0.3]
            \node[p] (1) at (0, -1) {};
            \draw[thick] (0, 0) circle (0.15cm);
            \draw[rotate=45, thick] (0,-0.15) -- (0,0.15);
            \draw[rotate=45, thick] (-0.15,0) -- (0.15, 0);
            \draw (1) .. controls (-1,-1) and (-1,1) ..  node [midway, left]{$i$}(0,1);
            \draw (1) .. controls (1,-1) and (1,1) ..  (0,1);
            \draw[red] (1) .. controls (-0.5,-1) and (-0.5,0) .. (-0.15,0);
            \draw[red] (1) .. controls (0.5,-1) and (0.5,0) .. node [midway,left]{$\t$}(0.15,0);
            \end{tikzpicture}
  $\rightsquigarrow$
  \begin{tikzpicture}[baseline={([yshift=-.5ex]current bounding box.center)}, scale=1.5]
            \tikzstyle{p}=[circle,fill, scale=0.3]
            \node[p] (1) at (0, -1) {};
            \draw[thick] (0, 0) circle (0.15cm);
            \draw[rotate=45, thick] (0,-0.15) -- (0,0.15);
            \draw[rotate=45, thick] (-0.15,0) -- (0.15, 0);
            \draw (1) .. controls (-1,-1) and (-1,1) .. node [left]{$i$}(0,1);
            \draw (1) .. controls (1,-1) and (1,1) .. (0,1);
            \draw[red] (0,-0.15) .. controls (0,-0.5) and (0.5,-0.25) .. node [below]{$\t'$}(0.5,0);
            \draw[red] (0,0.15) .. controls (0,0.5) and (0.5,0.25) .. (0.5,0);
            \end{tikzpicture}
 \caption{Type ii mutation}
 \label{Fig::Mutation2}
 \end{figure}
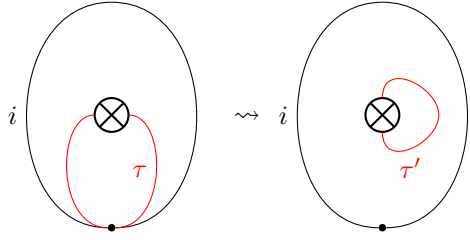

    \item If a one-sided closed curve $t$ is in the triangulation with boundary $i$ as in Figure \ref{Fig::Mutation3}, the relation is
    \begin{equation} \label{Eq::Mutation2}
    x_tx_t' = x_i.
    \end{equation}
    
\begin{figure}[h]
 \centering
 \begin{tikzpicture} [baseline={([yshift=-.5ex]current bounding box.center)}, scale=1.5]
            \tikzstyle{p}=[circle,fill, scale=0.3]
            \node[p] (1) at (0, -1) {};
            \draw[thick] (0, 0) circle (0.15cm);
            \draw[rotate=45, thick] (0,-0.15) -- (0,0.15);
            \draw[rotate=45, thick] (-0.15,0) -- (0.15, 0);
            \draw (1) .. controls (-1,-1) and (-1,1) .. node [left]{$i$}(0,1);
            \draw (1) .. controls (1,-1) and (1,1) .. (0,1);
            \draw[red] (0,-0.15) .. controls (0,-0.5) and (0.5,-0.25) .. node [below]{$\t$}(0.5,0);
            \draw[red] (0,0.15) .. controls (0,0.5) and (0.5,0.25) .. (0.5,0);
            \end{tikzpicture}
 $\rightsquigarrow$
 \begin{tikzpicture} [baseline={([yshift=-.5ex]current bounding box.center)}, scale=1.5]
            \tikzstyle{p}=[circle,fill, scale=0.3]
            \node[p] (1) at (0, -1) {};
            \draw[thick] (0, 0) circle (0.15cm);
            \draw[rotate=45, thick] (0,-0.15) -- (0,0.15);
            \draw[rotate=45, thick] (-0.15,0) -- (0.15, 0);
            \draw (1) .. controls (-1,-1) and (-1,1) .. node [left]{$i$}(0,1);
            \draw (1) .. controls (1,-1) and (1,1) .. (0,1);
            \draw[red] (0,-0.15) .. controls (0,-0.5) and (0.5,-0.25) .. node [below]{$\t'$}(0.5,0);
            \draw[red] (0,0.15) .. controls (0,0.5) and (0.5,0.25) .. (0.5,0);
            \end{tikzpicture}
 \caption{Type iii mutation}
 \label{Fig::Mutation3}
 \end{figure}
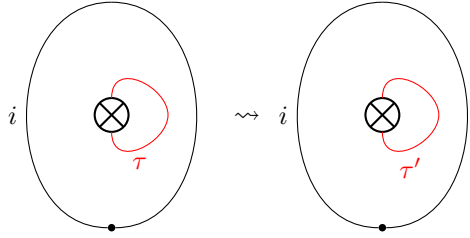
    
    
    \item  Let $t$ be the arc enclosing a quasi-arc $i$ and part of a triangle with other sides $j$ and $k$, as Figure \ref{Fig::Mutation4} shows, then the relation is 

    \begin{equation} \label{Eq::Mutation4}
    x_tx_t' = (x_j+x_k)^2+x_i^2x_jx_k.
    \end{equation}

\begin{figure}
\begin{tikzpicture}[baseline={([yshift=-.5ex]current bounding box.center)}, scale=1.5]
            \tikzstyle{p}=[circle,fill, scale=0.3]
            \node[p] (1) at (0, -1) {};
            \node[p] (2) at (0, 1) {};
            \draw[thick] (0, 0) circle (0.15cm);
            \draw[rotate=45, thick] (0,-0.15) -- (0,0.15);
            \draw[rotate=45, thick] (-0.15,0) -- (0.15, 0);
            \draw (1) .. controls (-1,-1) and (-1,1) .. node [left] {$j$} (2);
            \draw (1) .. controls (1,-1) and (1,1) .. node [right] {$k$}(2);
            \draw (0,-0.15) .. controls (0,-0.5) and (0.5,-0.25) ..node [midway,below] {$i$} (0.5,0);
            \draw (0,0.15) .. controls (0,0.5) and (0.5,0.25) .. (0.5,0);
            
            \draw[red] (2) .. controls (-0.7,0.2) and (-0.9,-0.7) ..node [midway, right] {$\t$} (0,-0.7);
            \draw[red] (2) .. controls (0.7,0.2) and (0.9,-0.7) .. (0,-0.7);
            \end{tikzpicture}
 $\rightsquigarrow$
\begin{tikzpicture}[baseline={([yshift=-.5ex]current bounding box.center)}, scale=1.5]
            \tikzstyle{p}=[circle,fill, scale=0.3]
            \node[p] (1) at (0, -1) {};
            \node[p] (2) at (0, 1) {};
            \draw[thick] (0, 0) circle (0.15cm);
            \draw[rotate=45, thick] (0,-0.15) -- (0,0.15);
            \draw[rotate=45, thick] (-0.15,0) -- (0.15, 0);
            \draw (1) .. controls (-1,-1) and (-1,1) .. node [left] {$j$} (2);
            \draw (1) .. controls (1,-1) and (1,1) .. node [right] {$k$}(2);
            \draw (0,-0.15) .. controls (0,-0.5) and (0.5,-0.25) ..node [midway,below] {$i$} (0.5,0);
            \draw (0,0.15) .. controls (0,0.5) and (0.5,0.25) .. (0.5,0);
            
            \draw[red] (1) .. controls (-0.7,-0.2) and (-0.9,0.7) ..node [midway, right] {$\t'$} (0,0.7);
            \draw[red] (1) .. controls (0.7,-0.2) and (0.9,0.7) .. (0,0.7);
            \end{tikzpicture}
             \caption{Type iv mutation}
 \label{Fig::Mutation4}
 \end{figure}
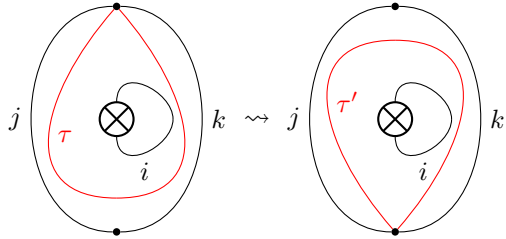

 \end{enumerate}
 
In general, a (quasi-)\textit{seed} is obtained by a sequence of mutations from the initial (quasi-)seed.
\end{definition}

\begin{remark}
Remark that the mutation of an arc of type \eqref{type2} gives a quasi-arc and vice versa. Besides, the type of a regular arc may change by performing flips on this arc or around it.
\end{remark}

The definition of (quasi)-seeds and quasi-mutations lead naturally to the definition quasi-cluster algebra.

\begin{definition}
Let $\X$ denote the union of all clusters obtained by successive quasi-mutation from an initial seed $(\x, \y, T)$. Define the \emph{quasi-cluster algebra} to be the polynomial algebra in variables $\X$ with coefficients in $\mathbb{Z}\mathbb{P}$ i.e. 

$$\A(\x,T) = \mathbb{Z}\mathbb{P}[\X]$$
\end{definition}

From these quasi-mutations, we can create something called the \emph{flip graph}, where the vertices are (quasi-)seeds and the edges are a single quasi-mutation. 

\begin{proposition}\label{prop::flipgraph}
\cite{DP15} The flip graph of quasi-triangulations of a marked surface $\SM$ is connected. 
\end{proposition}

\begin{crl}\label{crl::numberofquasiarcs}
The number of arcs in every quasi-triangulation of $\SM$ is an invariant of the marked surface $\SM$. 
\end{crl}
\end{subsection}

Mimicking the number of arcs formula in the orientable case, we have an explicit formula for the number of arcs in a quasi-triangulation of a non-orientable marked surface \cite{FST08}. This result was alluded to in \cite{DP15}, but we provide the formula and an original proof.

\begin{proposition} \label{prop::numofarcsformula}
Let $\SM$ be a non-orientable surface and let $T$ be a quasi-triangulation of $\SM$ without quasi-arcs. Then the number of arcs in a triangulation is given by
$$n = 3g+3b+3p+c-6$$
where $g$ is the genus of $S$ i.e. the number of cross-caps used to construct $S$, $b$ is the number of boundary components of $S$, $p$ is the number of punctures of $S$ and $k$ is the number of marked points on the boundary of $S$. 
\end{proposition}

\begin{proof}
In order to prove our formula, we use facts about the Euler characteristic of surfaces. For any surface $X$, we have that the Euler characteristic is given by

$$\chi(X) = f-e+v$$
where $f$ is the number of faces, $e$ is the number of edges and $t$ is the number of vertices in a triangulation of $X$. Using Proposition \ref{prop::eulerchar}, when $X$ is a non-orientable surface represented as a disk with $g$ cross-caps, the Euler characteristic is given by 
$$\chi(X) = 2-g.$$

Suppose our marked surface $\SM$ has $b$ boundary components, then the Euler characteristic is given by $\chi(S) = 2-g-b$ as creating one boundary component reduces the number of faces by 1.  

Consider the set of marked points $M$ of $S$ and let $c$ be the number of marked points on the boundary and $p$ be the number of interior marked points or punctures. Then the number of vertices in the quasi-triangulation $T$ of $\SM$ is given by $v= |M| = c+p$. Furthermore, if we let $n$ be the number of internal arcs in $T$, the number of edges $e$ in $T$ is the sum of the number of internal arcs and the number of boundary arcs i.e. $e = c+n$, as $c$ is the number of marked points on the boundary of $S$.  

Now, we need to compute the number of faces $f$. We proceed by induction of $p$ the number of punctures in $\SM$. When $p=0$, then there are no self-folded triangles and each triangle has three distinct sides. So each of the $n$ internal arcs is contained in exactly two triangles and each of the $c$ boundary segments lie in exactly one triangle. In this case, we have $3f=2n+c$.  

So, when $p=0$, we have that 

\begin{align*}
    \chi(S) &= f=e+v = \frac{2n+c}{3}- (c+n) + (c+p) \\
    &= \frac{2n+c}{3}-c-n+c+0\\
    &= \frac{2n+c}{3}-n
\end{align*}

Equating this to our other formula for the Euler characteristic based on the number of cross-caps, $g$, and the number of boundary components, $b$, we have 

$$2-g-b = \chi(S) = \frac{2n+c}{3}-n$$
Solving for $n$, we obtain our formula for the number of arcs: $n = 3g+3b+c-6$ when $p=0$.  

Suppose $p>0$ and let $x$ be a puncture we add to a triangulation $T$ of $\S$. To complete $T$ to a triangulation of the new punctured surface, we only need to add 3 arcs. Namely, since $x$ belongs to some triangulation $\Delta$ in $T$, we complete the triangulation by connecting $x$ to the three vertices of $\Delta$. This tells us that adding one puncture increases $n$ by 3 giving our formula $n = 3g+3b+3p+c-6$.
\end{proof}

By Proposition \ref{prop::flipgraph} and Corollary \ref{crl::numberofquasiarcs}, all quasi-triangulations are connected via a sequence of flips and the number of arcs in any quasi-triangulation is equal. This gives that Proposition \ref{prop::numofarcsformula} holds for any quasi-triangulation - namely, triangulations involving quasi-arcs. 

\subsection{Quiver Associated to an Orientable Surface}
Cluster algebras are studied in a more general frame using quivers, that we introduce in this section.

A \dfn{quiver} $Q=(Q_0, Q_1, s, t)$ is a quadruplet where: \begin{itemize}
    \item $Q_0$ is a set of \dfn{vertices} which contains a subset $F$ of \dfn{frozen vertices};
    \item $Q_1$ is a set of \dfn{arrows};
    \item $s$ and $t$ are two maps that associate to each arrow its \dfn{source} and its \dfn{target}.
\end{itemize}

For simplicity's sake, for a quiver with $n+m$ vertices and $m$ frozen vertices, we denote $Q_0 \setminus F = \{ 1, 2, \dots, n \}$ and $F = \{ n+1, n+2, \dots, m\}$.
Moreover, we agree to box frozen vertices when we draw a quiver.

\begin{example} \label{Ex::quiver}
Consider $Q$ the quiver shown in Figure \ref{Fig::Quiver}. The set of vertices in $Q$ is $Q_0 = \{ 1, 2, 3, 4, 5, 6, 7, 8, 9 \}$, the set of its frozen vertices is $F = \{ 7, 8, 9 \}$ and the set of its arrows is $Q_1 = \{ \alpha_1, \alpha_2, \alpha_3, \alpha_4, \alpha_5, \alpha_6, \beta_1, \beta_2, \beta_3, 
\beta_4,
\beta_5,
\gamma_1,
\gamma_2,
\gamma_3,
\gamma_4\}$.
\begin{figure}[ht]
\centering
\begin{tikzpicture}[scale = 1]
    \def \n {6}
    \def \radius {2cm}
    \def \margin {5} 
    
    \foreach \s in {1,...,\n}
    { \node (\s) at ({-360/\n * (\s - 1) +90}:\radius) {$\s$}; }
    \draw [->] (1) -- node[below right] {$\alpha_2$ }(6);
    \draw [->] (1) -- node[below left] {$\beta_2$} (2);
    \draw [->] (5) -- node[left] {$\beta_4$} (6);
    \draw [->] (2) -- node[above] {$\beta_3$} (5);
    \draw [->] (3) -- node[right] {$\alpha_5$} (2);
    \draw [->] (5) -- node[above] {$\g_3$} (3);
    \draw [->] (4) -- node[below left] {$\g_2$} (5);
    \draw [->] (6) -- node[near start, above right] {$\alpha_3$}  (4);
    \draw [->] (4) -- node[near start, above left] {$\alpha_4$} (3);
    \foreach \s in {7,8,9}
    { \node (\s) at ({-360*2/\n * (\s - 1) +30}:1.75*\radius) {$\boxed{\s}$}; }
    \draw [->] (9) -- node[above] {$\beta_1$} (1);
    \draw [->] (6) -- node[below] {$\beta_5$} (9);
    \draw [->] (7) -- node[above] {$\alpha_1$} (1);
    \draw [->] (2) -- node[below] {$\alpha_6$} (7);
    \draw [->] (8) -- node[left] {$\g_1$} (4);
    \draw [->] (3) -- node[right] {$\g_4$} (8);
\end{tikzpicture}
\caption{Quiver}
\label{Fig::Quiver}
\end{figure}

\end{example}

Let $Q$ be a quiver without loop nor $2$-cycle and let $k$ be a non-frozen vertex of $Q$. The \dfn{mutation} of $Q$ in the direction of $k$ transforms $Q$ into a new quiver $\mu_k(Q)$ by performing the following steps: \begin{enumerate}
    \item \label{step1} for all path $i \xrightarrow{} k \xrightarrow{} j$, add an arrow $i \rightarrow j$;
    \item inverse all arrows incident to $k$;
    \item remove one by one all $2$-cycles created at step \ref{step1}.
\end{enumerate}

Note that mutation is an involutive process i.e. $\mu_k^2(Q) = Q$, \cite{FZ02}.

A \dfn{path} in quiver is a sequence $(\alpha_1, \alpha_2, \dots, \alpha_m )$ of arrows in $Q$ such that $t(\alpha_k) = s(\alpha_{k+1})$ for all $k=1, 2, \dots, m-1$.
If either $(\alpha_1, \alpha_2, \dots, \alpha_m )$ or $(\alpha_m,  \dots, \alpha_2, \alpha_1 )$ is a path in $Q$, we denote this path both by $[\alpha_1 \alpha_2 \dots \alpha_m]$ and $[\alpha_m  \dots \alpha_2 \alpha_1 ]$. Similarly, if $p_1$ and $p_2$ are paths in $Q$ and if either $p_1p_2$ or $p_2p_1$ is also a path in $Q$, we denote this longer path by $[p_1p_2]$ or $[p_2p_1]$.

\begin{example}
Consider the paths $p_1 = \beta_1 \beta_2$ and $p_2 = \beta_3 \beta_4 \beta_5$. Then, $[p_1p_2] = \beta_1 \beta_2 \dots \beta_5 = [p_2p_1]$. Check that this is true in the case of Example \ref{Ex::quiver}.
\end{example}

It is possible to associate a quiver $Q_T$ to a triangulation $T$ by a process described by Fomin, Shapiro and Thurston, \cite{FST08}. To this end, fix an orientation to the surface $\S$. The internal arcs of the triangulation correspond to vertices in the quiver and there is an arrow from the vertex associated to an arc $\t$ to the vertex associated to an arc $\t'$ if $\t$ and $\t'$ share a common endpoint and at this endpoint, $\t'$ follows $\t$ according to the fixed orientation of the surface $\S$. In other words, each mark point in $\M$ determines a path in $Q_T$.

\begin{example}
The quiver in Figure \ref{Fig::Quiver} corresponds to the triangulation in Figure \ref{Fig::Triangulation}. The paths in $\alpha_1 \alpha_2 \dots \alpha_6$, $\beta_1 \beta_2 \dots \beta_5$ and $\g_1 \g_2 \g_3 \g_4$ are given by each of the three marked points in $\M$.
\end{example}

\begin{remark}
In the definition of quivers associated to orientable surfaces, we rely on having a unique notion of clockwise and counterclockwise. In a non-orientable surface, we do not have such a notion which does not allow us to naturally extend the definition of a quiver associated to a surface, as described in this section.
\end{remark}

We can now build a cluster algebra from a quiver instead of a triangulation of an orientable surface. Indeed, cluster algebras arising from quivers were studied before those arising from triangulations and they form larger class, \cite{FZ02}.

In the context of quivers, a \dfn{seed} is a triplet $(\x, \y, Q)$ where: \begin{enumerate}
    \item $\x$ is a cluster of $n$ variables;
    \item $\y$ is a set of $m$ coefficients;
    \item $Q$ is a quiver containing $n+m$ vertices and $m$ of them are frozen, that is are in $F$.
\end{enumerate}

We associate bijectively variables of $\x$ with non-frozen vertices of $Q_0 \setminus F$ and coefficients of $\y$ with frozen variables in $F$, the set of frozen vertices.

Just as flipping arcs in seed from a triangulation creates a new triangulation with a new cluster variable, mutating a quiver in a direction of a non-frozen vertex creates a new quiver with a new cluster variable.


\section{Quivers Arising from Non-Orientable Surfaces} \label{section::ourquiver}

In this section, we detail the construction of a quasi-cluster algebra from a quiver instead of a surface. We first describe how to associate what we call a partitioned quiver to a triangulation of a non-orientable unpunctured surface. After this, we define mutation on the level of our partitioned quiver and show that it is compatible with mutation on the level of the surface. As a consequence, we have that our partitioned quiver defines a quasi-cluster algebra. 

\subsection{Partitioned Quivers} \label{subsection::quiversqca}

\begin{definition}
A \dfn{partitioned} quiver is a pair $(Q,P)$ where $Q$ is a quiver whose set of arrows $Q_1$ admits a partition $P$ such that each subset of arrows in this partition is a path in $Q$; theses paths are called \dfn{$P$-paths}. 
\end{definition}

Let $(Q,P)$ be a partitioned quiver with $P$ and let $\alpha = \alpha_1 \alpha_2 \dots \alpha_a$ and $\beta = \beta_1 \beta_2 \dots \beta_b$ be paths in $Q$. 
If there is a path $p$ in $P$ such that $[\alpha \beta]$ is a subpath of $p$, then the paths $\alpha$ and $\beta$ are \dfn{$P$-concatenable} or, more precisely, \dfn{$p$-concatenable}.  

\begin{figure}[h]
 \begin{tikzcd}
    &9 \arrow[dd] & 1\arrow[drr,red]& 2 \arrow[l,red]\arrow[dr] &\\
    8\arrow[ur,blue] & &&&3\arrow[d]\arrow[ddll,red]\\
    &7 \arrow[ul]\arrow[dr] &&& 4\arrow[dl]\arrow[uul]\\
    &&6\arrow[r,red] \arrow[rru] &5
    \arrow[uuu,red]\arrow[ull]&\\
    \end{tikzcd}
    \caption{Partitioned quiver where the parts of the partition are indicated by color.}
    \label{fig::recoveringsurfacequiver}
    \end{figure}
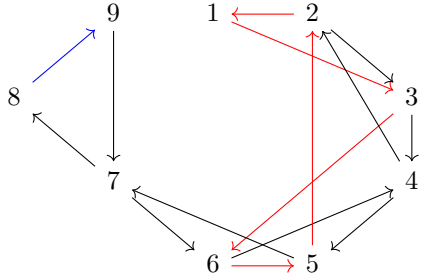

\begin{example}
Consider the partitioned quiver in Figure \ref{fig::recoveringsurfacequiver}. Then the path formed by the blue and black paths, say $P$, is a path itself. Hence the blue path and the black path are $P$- concatenable. 
\end{example}

We now construct a frozen quiver $Q_T$ from a triangulation $T$ of marked surface $\SM$. Even though the surface may be non-orientable, each boundary component corresponds to a removed disk i.e. is homeomorphic to $S^1$. Therefore, each boundary component is orientable; hence, we can fix an orientation to each of the boundary component. The arcs in $T$ are in bijection with the vertices of $Q_T$ and the boundary arc in $T$ are in bijection with the frozen vertices in $Q_T$.  
  
Suppose that $\t_1$ and $\t_2$ are two arcs cutting a common triangle $\Delta$ in $T$. Denote their corresponding vertices in $Q$ by $1$ and $2$ respectively. Consider their common endpoint, the marked point $i$. Suppose moreover that $\t_1$ precedes $\t_2$ according to the orientation of the boundary component on which lays $i$, see local picture at the marked point $i$ in Figure \ref{Fig::quiversetup}.  
 
\begin{figure}
\centering
\begin{tikzpicture}
    \tikzstyle{p}=[circle,fill, scale=0.3]
    \node[rotate=-45] [p] (1) at (0, 0) {$t$};
    
    \draw[rotate=-45] (-2,1) .. controls (-1,0) .. (0,0);
    \draw[rotate=-45] (0,0) .. controls (1,0) .. (2,1);
    \draw[rotate=-45] (1) -- node [midway,above] {$\t_1$} (1.5,1.5);
    \draw[rotate=-45] (1) -- node [midway,left] {$\t_2$} (-1.5, 1.5);
    \node[] at (1,1) {$\Delta$};
    \draw[rotate=-45] [latex-](-0.3,-0.3) -- (0.3,-0.3);
\end{tikzpicture}
\caption{The local configuration of the arcs $\t_1, \t_2$ in a regular triangle $\Delta$.} 
\label{Fig::quiversetup}
\end{figure}
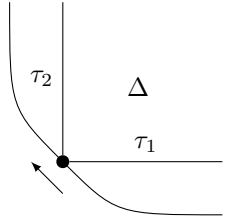

If $\Delta$ is a regular triangle, then there is an arrow from $1$ to $2$:
\[ \begin{tikzcd}
1 \arrow[r] & 2 
\end{tikzcd}\]

If $\Delta$ is an anti-self folded triangle containing the quasi-arc $\t_3$, then, $\t_1 = \t_2$ and we identify the vertices $1 = 2$. The local configuration for this case can be seen in Figure \ref{Fig::quiversetupquasi}. In that case, there are arrows from $1$ to $3$, from $3$ to itself and from $3$ to $2$, where $3$ is the vertex in $Q$ associated with $\t_3$:
\[ \begin{tikzcd}
1 \arrow[r, shift left] & 3 \arrow[l, shift left]
\end{tikzcd}\]

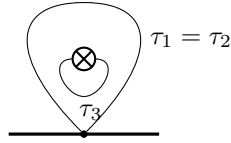
\begin{figure}
\centering

\begin{tikzpicture}
    \draw[very thick] (-1,-1) -- (1,-1);
    \tikzstyle{p}=[circle,fill, scale=0.3]
    \node[p] (1) at (0, -1) {};
    \draw[thick] (0, 0) circle (0.15cm);
    \draw[rotate=45, thick] (0,-0.15) -- (0,0.15);
    \draw[rotate=45, thick] (-0.15,0) -- (0.15, 0);
    \draw (1) .. controls (-1,0) and (-1,0.75) ..  (0, 0.75);
    \draw (1) .. controls (1,0) and (1,0.75) .. node [right] {$\t_1=\t_2$} (0, 0.75);
    \draw[rotate=-90] (0,-0.15) .. controls (0,-0.5) and (0.5,-0.25) .. (0.5,0);
    \draw[rotate=-90] (0,0.15) .. controls (0,0.5) and (0.5,0.25) .. node [very near end, below] {$\t_3$}(0.5,0);
    \end{tikzpicture}
\caption{The local configuration around a quasi-arc.} 
\label{Fig::quiversetupquasi}
\end{figure}

 \begin{figure}[h]
 \centering
 \begin{tikzpicture} [scale = 0.75]
    \draw[very thick] (0,0) circle (3cm) node[right] at (-3,0){$\t_6$};
    \draw[very thick] (0, 1.25) circle (0.75cm) node[left] at (0.75,1.25){$\t_7$};
    
    \draw[very thick] (0, -1.25) circle (0.25cm);
    \draw[shift={(0,-1.25)}, rotate=45] (0,-0.25) -- (0,0.25);
    \draw[shift={(0,-1.25)}, rotate=45] (-0.25,0) -- (0.25, 0);
    
    \tikzstyle{p}=[circle,fill, scale=0.3]
    \node[p] (1) at (0,0.5) {};
    \node[p] (2) at (0, 3) {};
    
    \draw (1) .. controls (-2, 0.5) and (-1.25, 2) .. node[near start, below] {$\t_1$} (2);
    
    \draw (1) .. controls (2,0.5) and (1.25,2) .. node[near start, below]{$\t_2$} (2);
    
    \draw (1) .. controls (-2,0) and (-2,-2.75) .. (0,-2.75);
    \draw (0,-2.75) .. controls (2,-2.75) and (3.25,0.5) .. node[midway, right]{$\t_3$}(2);
    
    \draw (1) .. controls (-3,-3) and (3,-3) .. node[midway, below] {$\t_4$}(1);
    
    \draw (0.25,-1.25) ..  controls (0.5,-1.25) and (0.45,-0.75) .. node[very near end, above]{$\t_5$} (0,-0.7);
    \draw (0,-0.7) .. controls (-0.45,-0.75) and (-0.5,-1.25)  .. (-0.25,-1.25); 
 \end{tikzpicture}
 \caption{An example of a quasi-triangulation of the annulus with one crosscap.}
 \label{Fig::Triangulation7arcs}
 \end{figure}
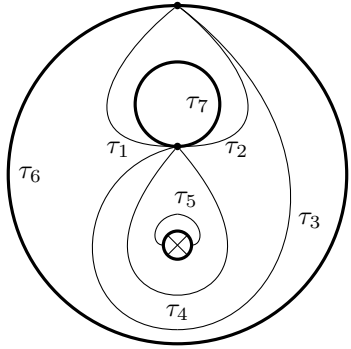
 
 \begin{example} \label{example:partitionedquiver}
In Figure \ref{Fig::Triangulation7arcs}, we have a quasi-triangulation of the annulus with one cross-cap and two marked points. We produce the quiver associated to this triangulation which is pictured in Figure \ref{Fig::quiver7arcs}. The set of vertices in $Q$ is $Q_0 = \{ 1, 2, 3, 4, 5, 6, 7\}$, the set of its frozen vertices is $F = \{ 6, 7 \}$ and the set of its arrows is $Q_1 = \{ \alpha_1, \alpha_2, \alpha_3, \alpha_4, \beta_1, \beta_2, \beta_3, \beta_4, \beta_5, \beta_6, \beta_7 \}$. Note that to differentiate the frozen vertices coming from boundary edges, we put the vertices 6,7 in squares. The partition $P_T$ associated to this triangulation and quiver has two parts: one coming from the marked point on the outer boundary component containing the paths $6 \to 3 \to 2 \to 1 \to 6$ and one coming from the marked point on the inner boundary component containing the paths $7 \to 1 \to 3 \to 4 \to 5 \to 4 \to 2 \to 7$.
\end{example} 
 
\begin{figure}
\begin{center}
\begin{tikzcd}
 & \boxed{7}   \arrow[d l]  \\
1 \arrow[d] \arrow[d r]& 2 \arrow[l]  \arrow[u] & 4 \arrow[l] \arrow[d, shift left]\\
\boxed{6} \arrow[r]  & 3  \arrow[r u] \arrow[u]& 5 \arrow[u, shift left]
\end{tikzcd}
\end{center}
\caption{Quiver arising from the surface in Figure \ref{Fig::Triangulation7arcs}.}
\label{Fig::quiver7arcs}
\end{figure}
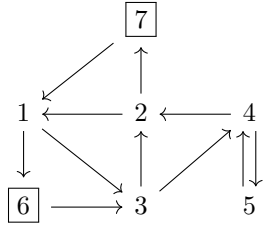

As alluded to in Example \ref{example:partitionedquiver}, the construction of $Q_T$ naturally gives a partition $P_T$ of the arrows of $Q_T$ based on the marked points of $(\S, \M)$. Indeed, each arrow of $Q_T$ is obtained by considering a marked point of $(\S, \M)$ and all arrows arising from the same marked point form a path following the orientation of the boundary component on which lay the marked points. In other words, each marked point gives rise to a path of arrows in $Q_T$ and each arrow is part of a single path; the set of these paths is denoted $P_T$ is there is a direct bijection between paths in $P_T$ and marked points.

Observe that fixing another orientation for a given boundary component will reverse the direction of all arrows given by marked points on this boundary component, but it will keep the partition the same. Thus, with our construction,  the obtained quiver $Q_T$ and set of paths $P_T$ are unique up to reversing the arrows forming paths associated with marked points all laying on the same boundary component.

\begin{remark} \label{rmk::Qlocal}
We can obtain $Q_T$ alternatively by constructing it locally around each internal arc $\t$, using Lemma \ref{lem::typearc}. Let $t$ be the vertex in $Q_T$ associated with $\tau$

\begin{enumerate} [i)] \setlength\itemsep{1em}
    \item If $\t_t$ is an arc of type \eqref{type1} as described in Lemma \ref{lem::typearc}, then $Q_T$ contains the following subgraph:
    
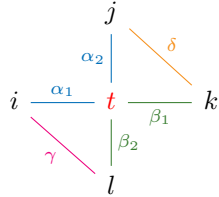
\[ \begin{tikzcd}
     & j \arrow[dr, dash, "\delta", BurntOrange] \\
     i  \arrow[r, dash, "\alpha_1", NavyBlue] \arrow[dr, dash, "\g", swap, RubineRed] & \textcolor{red}{t} \arrow[u, dash, "\alpha_2", NavyBlue] \arrow[r, "\beta_1", dash, swap, OliveGreen] & k \\
     & l \arrow[u, "\beta_2", dash, swap, OliveGreen]
    \end{tikzcd} \]
    with  the vertices $i$, $j$, $k$ and $l$  are respectively associated with the arcs $\t_i$, $\t_j$, $\t_k$, $\t_k$ and $\t_l$ of Figure \ref{Fig::Mutation1}. 
    
    Here, $[\alpha_1 \alpha_2]$, $[\beta_1 \beta_2] \not\in I$ while $[\alpha_1 \gamma], [\alpha_2 \delta], [\beta_1 \delta], [\beta_2 \gamma] \in I$. Indeed, since $\alpha_1$ and $\gamma$ arise from two different endpoints of the arc $\t_i$, their composition is in the ideal $I$. The same argument applies to $\alpha_2$ and $\delta$, $\beta_1$ and $\gamma$, and $\beta_2$ and $\delta$ by considerating the endpoints of $\t_j$, $\t_k$ and $\t_l$ respectively.
    \item  If $\t_t$ is an arc of type \eqref{type2} as described in Lemma \ref{lem::typearc}, then $Q_T$ contains the following subquiver:
    \[ \begin{tikzcd}
    i  \arrow[r, shift left, "\alpha_1", NavyBlue] & t \arrow[l, shift left, "\alpha_3", NavyBlue] \arrow[loop right, "\alpha_2", NavyBlue]
    \end{tikzcd}\]
     with $[\alpha_1 \alpha_2 \alpha_3]$ a subpath of a path in $P_T$.
    Here the vertex $i$ is associated with the arc $\t_i$ of Figure \ref{Fig::Mutation2}. 
    \item  If $\t_t$ is an arc of type \eqref{type3} as described in Lemma \ref{lem::typearc}, then $Q_T$ contains the following subgraph:
    
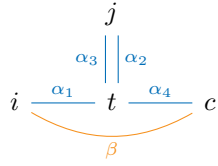
\[ \begin{tikzcd}
    & j \arrow[d, shift left, dash, "\alpha_2", NavyBlue]\\
    i  \arrow[r, dash, "\alpha_1", NavyBlue] \arrow[rr, dash, bend right, "\beta", swap, BurntOrange] & t \arrow[r, dash, "\alpha_4", NavyBlue] \arrow[u, shift left, dash, "\alpha_3", NavyBlue] & c
    \end{tikzcd}\]
    with $[\alpha_1 \alpha_2 \alpha_3 \alpha_4]$ a subpath of a path in $P_T$. Here, the vertices $i$, $j$, and $k$  are respectively associated with the arcs $\t_i$, $\t_j$, $\t_k$, $\t_k$ and $\t_l$ of Figure \ref{Fig::Mutation3}.
    \item If $\t_t$ is a quasi-arc enclosed by the arc $\t_i$, then $Q_T$ contains the following subquiver:
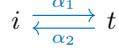
\[ \begin{tikzcd}
i  \arrow[r, shift left, "\alpha_1", NavyBlue] & t \arrow[l, shift left, "\alpha_2", NavyBlue]
\end{tikzcd}\]
with $[\alpha_1 \alpha_2]$ a subpath of a path in $P_T$.
Here $\t_t$ and  $\t_i$ are respectively associated with the vertices $t$ and $i$ of Figure \ref{Fig::Mutation4}.
\end{enumerate}

\end{remark}

Just as Lemma \ref{lem::typearc} allows us to sort arcs of a triangulation $T$ in four types depending on the local configuration of $T$, Remark \ref{rmk::Qlocal} allows us to sort vertices in four types depending of the local configuration of $Q_T$ - one type of vertices for each type of regular arc and one type of vertices for the quasi-arcs.


\subsection{Mutation of Partitioned Quivers} \label{subsection::mutquiver}

Before defining the mutation of partitioned quivers, it is essential to observe the uniqueness  of the local configuration of $Q_T$ around each vertices. In particular, we want to show that the arrows $\g$ and $\delta$ given in Remark \ref{rmk::Qlocal} for arcs of type \eqref{type1} are unique.

\begin{lemma}\label{lemm::arcsdeltagamma}
Let $T$ be a triangulation of a marked surface $(\S, \M)$ and let $(Q_T, P_T)$ be its partitioned quiver. 
Consider $\t_t$ be an arc of type \eqref{type1} and let $t$ be the associated vertex in $Q_T$. Then, there exist $2$-paths $[\alpha_1 \alpha_2]$ and $[\beta_1 \beta_2]$ through $t$ and a unique pair of arrows $\{\g, \delta \}$ such that: \begin{itemize}
    \item $[\alpha_1 \alpha_2], [\beta_1 \beta_2] \not\in I$,
    \item $\g$ and $\delta$ arc concatenable in $Q_T$ with both $[\alpha_1 \alpha_2]$ and $[\beta_1 \beta_2]$
    \item $[\alpha_1 \gamma], [\alpha_2 \delta], [\beta_1 \delta], [\beta_2 \gamma] \in I$.
\end{itemize}
\end{lemma}

\begin{proof}
We know from the construction of the quiver $Q_T$ out of the triangulation $T$ that there exist such arrows $\alpha_1, \alpha_2, \beta_1, \beta_2, \delta, \gamma$, see Remark \ref{rmk::Qlocal} for more explanations.

Our aim is to prove the unicity for the listed properties of arrows $\delta$ and $\gamma$. Suppose there exists a second arrow $\delta'$ such that $\{ s(\delta'), t(\delta')\} = \{ s(\delta), t(\delta)\}$ and such that $[\alpha_2 \delta], [\beta_1 \delta] \in I$:
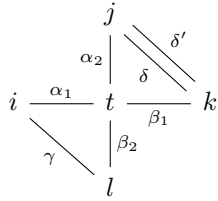
\[ \begin{tikzcd}
     & j \arrow[dr, dash, "\delta'", shift left] \arrow[dr, dash, "\delta", shift right, swap] \\
     i  \arrow[r, dash, "\alpha_1"] \arrow[dr, dash, "\g", swap, ] & t \arrow[u, dash, "\alpha_2"] \arrow[r, "\beta_1", dash, swap] & k \\
     & l \arrow[u, "\beta_2", dash, swap]
    \end{tikzcd} \]


We now prove unicity: suppose that there is another arrow $\delta'$ whose source and target are $x$ and $y$.
\end{proof}

Lemma \ref{lemm::arcsdeltagamma} allows us to define the mutation of partitioned quiver on any non-frozen vertices. To do so, we use the local configuration of $Q_T$ given in Remark \ref{rmk::Qlocal}. Let $t$ be a non-frozen vertex of $(Q_T,P)$. The mutation of the partitioned quiver $(Q_T,P)$ in direction $t$, transforms it into a new partitioned quiver $\left(\mu_t(Q_T), \mu_t(P)\right)$ as described below. 


\subsubsection{Mutation in Type \eqref{type1}} \label{subsubsection::type1mutationdefinition}
Suppose that the local configuration around $t$ is of type \eqref{type1} in Remark \ref{rmk::Qlocal}.
Denote by $p_a$, $p_b$, $p_c$ and $p_d$ the paths in the partition $P$ in which lies respectively $[\alpha_1 \alpha_2]$, $[\beta_1 \beta_2]$, $\gamma$ and $\beta$.

The mutated quiver $\mu_t(Q)$ is obtained from $Q$ by performing the following steps: \begin{enumerate}
	\item remove the arrows $\alpha_1$ and $\alpha_2$ and replace it by an arrow \[ s([\alpha_1 \alpha_2]) \xrightarrow{\alpha} t([\alpha_1 \alpha_2]); \]
    \item remove the arrows $\beta_1$, and $\beta_2$ and replace it by an arrow \[s\left([\beta_1 \beta_2]\right) \xrightarrow{\beta} t([\beta_1 \beta_2]);\]
    \item remove the arrow $\delta$ and replace it by a path of length two \[ [\delta_1 \delta_2] = s(\delta) \to t \to t(\delta);\]
    \item remove the arrow $\g$ and replace it by a path of length two \[ [\g_1 \g_2] = s(\g) \to t \to t(\g).\]
    

\end{enumerate}
The underlying graph of $Q_T$ will therefore change locally as follows:

\begin{figure}[h]
\[\begin{tikzcd}
     & j \arrow[dr, dash, "\delta", BurntOrange] \\
     i  \arrow[r, dash, "\alpha_1", NavyBlue] \arrow[dr, dash, "\g", swap, RubineRed] & \textcolor{red}{t} \arrow[u, dash, "\alpha_2", NavyBlue] \arrow[r, "\beta_1", dash, swap, OliveGreen] & k \\
     & d \arrow[u, "\beta_2", dash, swap, OliveGreen]
    \end{tikzcd}
\leftrightsquigarrow
\begin{tikzcd}
 & j \arrow[dl, dash, "\alpha", swap, NavyBlue] \\
 i \arrow[r, "\g_1", swap, , dash, RubineRed] & t \arrow[u, "\delta_2", swap, dash, BurntOrange]  \arrow[d, "\g_2", swap, dash, RubineRed] & k \arrow[l, "\delta_1", swap, dash, BurntOrange] \\
 & l \arrow[ur, dash, "\beta", swap, OliveGreen] 
\end{tikzcd}\]
\caption{ }
\label{fig::quivermutation1}
\end{figure}
Note that the orientation of arrows incident to $t$ may or may not be reversed by the mutation; it only depends on the orientation of boundary components on which lay initial and new endpoints of $\t_t$.

The mutated partition of $Q_1$ $\mu_t(P_T)$ is obtained from $P_T$ by performing the following steps:
\begin{enumerate}
    \item remove the subpath $[\alpha_1$ $\alpha_2]$ from $p_a$ and replace it by $\alpha$;
    \item remove the subpath $[\beta_1$ $\beta_2]$ from $p_b$ and replace it by $\beta$;
    \item remove the arrow $\gamma$ from $p_c$ and replace it by [$\g_1$ $\g_2$];
    \item remove the arrow $\delta$ from $p_d$ and replace it by [$\delta_1$ $\delta_2$].
\end{enumerate}


\subsubsection{Mutation in Type \eqref{type2}} \label{subsubsection::type2mutationdefinition}
Suppose that the local configuration around $t$ is of type \eqref{type2} in Remark \ref{rmk::Qlocal}. Denote by $p_a$ the path in the partition $P$ in which lies $[\alpha_1 \alpha_2 \alpha_3]$. The mutated quiver $\mu_t(Q)$ is obtained from $Q$ by removing the arrow $\alpha_3$.

The underlying graph of $Q_T$ will therefore change locally as follows:
\[ \label{quiverMutation2} \begin{tikzcd}
    i \arrow[r, shift left, "\alpha_1", NavyBlue] & t \arrow[l, shift left, "\alpha_3", NavyBlue] \arrow[loop right, "\alpha_2", NavyBlue]
    \end{tikzcd}
    \leftrightsquigarrow
     \begin{tikzcd}
    i \arrow[r, shift left, "\alpha_1", NavyBlue] & t \arrow[l, shift left, "\alpha_2", NavyBlue]
    \end{tikzcd}\]
    
The mutated partition of $Q_1$ $\mu_t(P_T)$ is obtained from $P_T$ by removing the arrow $\alpha_3$ from $p_a$.


\subsubsection{Mutation in Type \eqref{type3}}\label{subsubsection::type3mutationdefinition}
Suppose that the local configuration around $t$ is of type \eqref{type3} in Remark \ref{rmk::Qlocal}. Denote by $p_a$ and $p_b$ the paths in the partition $P$ in which lies respectively $[\alpha_1 \alpha_2 \alpha_3 \alpha_4]$ and $[\beta]$. The mutated quiver $\mu_t(Q)$ is obtained from $Q$ by performing the following steps: 

\begin{enumerate}
	\item remove the arrows $\alpha_1$, $\alpha_2$, $\alpha_3$ and $\alpha_4$ and replace it by an arrow \[ s([\alpha_1 \alpha_2 \alpha_3 \alpha_4]) \xrightarrow{\alpha} t([\alpha_1 \alpha_2 \alpha_3 \alpha_4]) \]
	\item remove the arrow $\beta$ and replace it by a length four path \[ [\beta_1 \beta_2 \beta_3 \beta_4] = s(\beta) \to t \to j \to t \to t(\beta) \]
\end{enumerate}

The underlying graph of $Q_T$ will therefore change locally as follows:

\[  \begin{tikzcd}\label{fig::quivermutation3}
    & j \arrow[d, shift left, dash, "\alpha_2", NavyBlue]\\
    i \arrow[r, dash, "\alpha_1", NavyBlue] \arrow[rr, dash, bend right, "\beta", swap, BurntOrange] & t \arrow[r, dash, "\alpha_4", NavyBlue] \arrow[u, shift left, dash, "\alpha_3", NavyBlue] & k
    \end{tikzcd}
    \leftrightsquigarrow
    \begin{tikzcd}
    & j \arrow[d, shift left, dash, "\beta_2", BurntOrange]\\
    i \arrow[r, dash, "\beta_1", BurntOrange] \arrow[rr, dash, bend right, "\alpha", swap, NavyBlue] & t \arrow[r, dash, "\beta_4", BurntOrange] \arrow[u, shift left, dash, "\beta_3", BurntOrange] & k
    \end{tikzcd}\]
    
The mutated partition of $Q_1$ $\mu_t(P_T)$ is obtained from $P_T$ by performing the following steps:
\begin{enumerate}
    \item remove the subpath $[\alpha_1 \alpha_2 \alpha_3 \alpha_4]$ from $p_a$ and replace it by $\alpha$;
    \item remove the arrow $\beta$ from $p_b$ and replace it by $[\beta_1 \beta_2 \beta_3 \beta_4]$.
\end{enumerate}


\subsubsection{Mutation in Type \eqref{type4}}\label{subsubsection::type4mutationdefinition}
Suppose that the local configuration around $t$ is of type \eqref{type4} in Remark \ref{rmk::Qlocal}. Denote by $p_a$ and $p_b$ the paths in the partition $p_a$ in which lies [$\alpha_1$ $\alpha_2$ $\alpha_3$ $\alpha_4$], and $p_b$ in which lies [$\beta$]. The mutated quiver $\mu_t(Q)$ is obtained from $Q$ by adding a loop $t \xrightarrow{\alpha_3} t$

The underlying graph of $Q_T$ will therefore change locally as follows:
\[ \begin{tikzcd}\label{fig::quivermutation4}
    i \arrow[r, shift left, "\alpha_1", NavyBlue] & t \arrow[l, shift left, "\alpha_2", NavyBlue]
    \end{tikzcd}
    \leftrightsquigarrow
    \begin{tikzcd}
    i \arrow[r, shift left, "\alpha_1", NavyBlue] & t \arrow[l, shift left, "\alpha_3", NavyBlue] \arrow[loop right, "\alpha_2", NavyBlue]
    \end{tikzcd}.
\]

The mutated partition of $Q_1$ $\mu_t(P_T)$ is obtained from $P_T$ by adding the arrow $\alpha_3$ to $p_a$ between $\alpha_1$ and $\alpha_2$.

\begin{example}
Consider the quiver in Figure \ref{Fig::quiver7arcs} given by the surface in Figure \ref{Fig::Triangulation7arcs}. Suppose we wish to mutate at $\gamma_2$. Without the additional data of the partition of the arrows, notice there are two different configurations of mutated quivers that are illustrated in Figure \ref{Fig::example_mutate}.
\end{example}

\begin{figure}
\begin{center}
\begin{tikzcd}
 & \boxed{7}   \arrow[d] \arrow[d l]  \\
1 \arrow[d] \arrow[r]& 2   \arrow[r] \arrow[d]   & 4 \arrow[bend right=30, ll] \arrow[d, shift left] \arrow[l u] \\
\boxed{6} \arrow[r]  & 3  \arrow[bend right=30, u u] \arrow[r u] & 5 \arrow[u, shift left]\\
\end{tikzcd}
\hspace{1cm}
\begin{tikzcd}
 & \boxed{7}   \arrow[d r]  \\
1 \arrow[r] \arrow[d r]& 2 \arrow[d]  \arrow[u] & 4 \arrow[l] \arrow[d, shift left]\\
\boxed{6} \arrow[u] \arrow[r] & 3  \arrow[l u, shift left] & 5 \arrow[u, shift left]
\end{tikzcd}
\end{center}
\caption{Two possible mutations of quivers that could be obtained from the surface in Figure \ref{Fig::Triangulation7arcs} if the data of the partition of the arrows is not given.}
\label{Fig::example_mutate}
\end{figure}
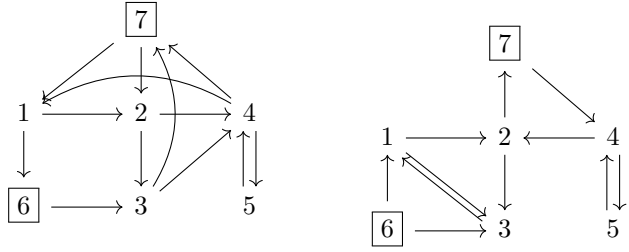

\subsection{Properties of Mutation of Partitioned Quivers} \label{subsec::mutquiver}

In this subsection, we explore properties of partitioned quiver mutation. 

Now, let's verify that, for any non-frozen vertex $t$ of a triangulation $T$ associated with the partitioned quiver $(Q_T, P)$, $\mu_t(P)$ is actually a partition of the arrows of $\mu_t(Q)$. Note that in each local configuration shown in Figure \ref{Fig::surfacelocalconfigs}, each arrow of $\mu_t(Q)$ belongs to a new part of the partition $\mu_t(P)$. Namely, by definition of mutation in each type, there is a unique assignment to each created arrow in $\mu_t(Q)$ to a part of $\mu_t(P)$, which gives that $\mu_t(P)$ is a partition on the set of arrows of $\mu_t(Q)$.

Another important characteristic of the mutation of quivers is that it is an involution in all types. 

\begin{theorem} \label{theorem::involution}
Let $T$ be a triangulation of a marked surface (orientable or not) and let $(Q_T,P_T)$ be the associated partitioned quiver. Consider $\t$ an arc in $T$ and $t$ the associated vertex in $Q_T$. Then,
\[ \mu_t^2(Q_T,P_T) = (Q_T,P_T).\]
\end{theorem}

\begin{proof}
In order to show that mutation is an involution, we separate the proof into the four types of partitioned quiver mutation defined in Definition \ref{def::quasimutation}.  
  
If $\t_t$ is an arc of type \eqref{type1} as described in Lemma \ref{lem::typearc} i.e. $\t_t$ is the diagonal of some quadrilateral as in Figure \ref{Fig::surfacelocalconfigs}. By Subsection \ref{subsubsection::type1mutationdefinition}, the mutated partition $\mu_t(P_T)$ is obtained from $P_T$ by removing the subpath $[\alpha_1 \alpha_2]$ and $[\beta_1 \beta_2]$ from $p_a, p_b$ and replace it by $\alpha$ and $\beta$ respectively, and also removing the arrow $\gamma$, $\delta$ from $p_\gamma, p_\delta$ and replace it by $[\gamma_1 \gamma_2]$ and $[\delta_1 \delta_2]$. The resulting partitioned quiver $\mu_t(P_T)$ then has the unique 2-subpath $[\gamma_1 \gamma_2]$ and $[\delta_1 \delta_2]$. Hence we perform the same algorithm again, removing the subpath $[\gamma_1 \gamma_2]$ and $[\delta_1 \delta_2]$ in $p_\gamma$ and $p_\delta$ and replace it by $\gamma$ and $\delta$ respectively, and also removing the arrow $\alpha$, $\beta$ from $p_\gamma, p_\delta$ and replace it by $[\gamma_1 \gamma_2]$ and $[\delta_1 \delta_2]$ respectively, which results from the original quiver back again.  
  
Suppose $\t_t$ is an arc of type \eqref{type2} i.e. $\t_t$ is the loop based at a marked point enclosing a cross cap as in Figure \ref{Fig::surfacelocalconfigs}. Suppose that $[\alpha_1\alpha_2]$ is a 2-path in the local picture associated to $\t_t$ in $(Q_T,P_T)$.  Then $\mu_t(Q_T,P_T)$ contains that subpath $[\alpha_1\alpha_3\alpha_2]$ where $\alpha_3$ is the loop around vertex $t$. Then $\t_t$ becomes an arc of \eqref{type2} and $\mu_t^2(Q_T,P_T)$ yields the 2-path $[\alpha_1\alpha_2]$ as desired.

Suppose $\t_t$ is an arc of type \eqref{type3} i.e. $\t_t$ is a quasi-arc enclosed by the loop $\t_i$ as in Figure \ref{Fig::surfacelocalconfigs}. Suppose that $[\alpha_1\alpha_2\alpha_3\alpha_4]$ is the 4-path starting between vertex $i$ and $y$ in $p_a$. After applying $\mu_t$, this 4-path becomes one arrow $\alpha$ between vertex $i$ and $y$ in $p_a$. Applying $\mu_t$ to the arrow $\alpha$ in $\mu_t(Q)$ in $p_a$ replaces this arrow with a 4-path in $p_a$ as desired. Suppose $\beta$ in $Q$ in $p_b$ is the arrow between the vertices $i$ and $y$. After applying $\mu_t$, this arrow $\beta$ becomes a 4-path $[\beta_1\beta_2\beta_3\beta_4]$ in $p_b$. Applying $\mu_t$ to the 4-path in $p_b$ replaces $[\beta_1\beta_2\beta_3\beta_4]$ in $p_b$ by a direct arrow between $i$ and $y$ as desired.

If $\t_t$ is an arc of type \eqref{type4} i.e. $\t_t$ is a loop based at a marked point traversing through the crosscap as in Figure \ref{Fig::surfacelocalconfigs}. Suppose that $[\alpha_1 \alpha_2 \alpha_3 \alpha_4]$ is the 4-path starting at vertex $y$ and ending at $x$ in $p_a$. After applying $\mu_t$, this 4-path becomes $[\beta_4 \beta_3 \beta_2 \beta_1]$ starting at vertex $x$ and ending at vertex $y$ in switches parts of the partition to $p_b$. Applying $\mu_t$ to the 4-path $[\beta_4 \beta_3 \beta_2 \beta_1]$ in $\mu_t(Q)$ in $p_b$, we obtain the 4-path $[\alpha_1 \alpha_2 \alpha_3 \alpha_4]$ in $p_a$ that we started with.  Suppose $\beta$ is the arrow starting at $x$ and ending at $y$ in $p_b$. By definition of $\mu_t$, $\beta$ in $p_b$ becomes the arrow $\alpha$ from $y$ to $x$ and switches parts of the partition to $p_a$. Applying $\mu_t$ to $\alpha$ in $\mu_t(Q)$ in $p_a$, we obtain $\beta$ from $x$ to $y$ in $p_b$ that we started with.
\end{proof}

\begin{theorem}\label{theorem::classical}
The partitioned quiver $(Q_T,P_T)$ generalized the classical quiver arising from cluster algebras from surfaces. Moreover, partitioned quiver mutation generalizes quiver mutation of type \eqref{type1} coincides with classical quiver mutation.  
\end{theorem}

\begin{proof}
Let $\SM$ be an orientable marked surface without punctures with a fixed orientation and let $T$ be a triangulation of $\SM$. Let $Q$ be the quiver defined by \cite{FST08} associated to $T$ and let $Q'$ be the quiver defined in Section \ref{section::ourquiver} associated to the internal arcs of $T$. Note that the number of vertices in $Q$ and $Q'$ correspond to the number of arcs in $T$, so $|Q| = |Q'|$. Hence, in order to show that the definition of our partitioned quiver $Q'$ is the same as the quiver $Q$ defined by \cite{FST08}, i.e. $Q = Q'$, we show that each vertex $v \in Q$ has the same local configuration as $v' \in Q'$.

 As $S$ is orientable, every internal arc $\t$ of $T$ is a regular arc between two marked points in $M$. Moreover, $\t$ cuts two regular triangles, call them $\Delta_1, \Delta_2$. Let $t$ be the vertex associated to $\t$; by the construction in \cite{FST08}, there are two arrows incident to $t$ for each of $\Delta_1$, call them $\alpha_1, \alpha_2$ ,and $\Delta_2$, call them $\beta_1, \beta_2$. So, there are vertices $i,j,k,l \in Q$ such that $i \xrightarrow{\alpha_1}t\xrightarrow{\alpha_2} j $ and $l \xrightarrow{\beta_1} t \xrightarrow{\beta_2} k$. This is the local description of the vertex $t \in Q$.

On the other hand, we obtain the local configuration for $t \in Q'$ by scanning at marked points. As $S$ is orientable, the orientation fixed on the surface induces an orientation on the boundary of $S$. Let $q,r$ be the two marked points that are the endpoints of $\t$. Since $\t$ cuts two regular triangles, there exist arcs $\t_j, \t_i$ that emanate from $q$ that are legs of triangles $\Delta_1, \Delta_2$ respectively. By our construction of $Q'$, this gives a subpath $i \xrightarrow{\alpha_1} t \xrightarrow{\alpha_2} j $ in $Q'$. At the marked point $r$, there exists arcs $\t_l, \t_k$ that emanate from $m$ that are legs of triangles $\Delta_1, \Delta_2$ respectively. By our construction of $Q'$, this gives a subpath $l \xrightarrow{\beta_1} t \xrightarrow{\beta_2} k $ in $Q'$. Therefore, we see that the local configuration of each vertex coincides and thus, $Q = Q'$.

We now show that mutation of type \eqref{type1} is the same as classical quiver mutation. Recall that classical quiver mutation at vertex $t \in Q$ is given by the following three step process:

\begin{enumerate}
    \item Reverse all arrows incident to $t \in Q$.
    \item For any two-path $s \to t \to u$, add an arrow directly from $s \to u$.
    \item Delete any created 2-cycles.
\end{enumerate}

Using the local configuration of $t \in Q$, we have vertices $i,j,k,l \in Q$ such that $i \xrightarrow{\alpha_1} t \xrightarrow{\alpha_2} j $ and $l \xrightarrow{\beta_1} t \xrightarrow{\beta_2} k$. According to classical quiver mutation, we add arrows $i \to j$ and $l \to k$. In addition to this, we reverse the orientation of $\alpha_1, \alpha_2, \beta_1, \beta_2$ to obtain, $i \leftarrow t \leftarrow j $ and $l \leftarrow t \leftarrow k$. To better visualize this mutation,

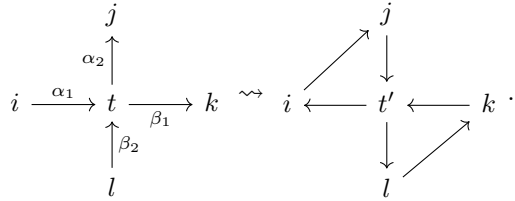
\[ \label{fig::quivermutationclassical} \begin{tikzcd}
 & j \\
 i \arrow[r, "\alpha_1"]  & t \arrow[u, "\alpha_2"] \arrow[r, "\beta_1", swap] & k \\
 & l \arrow[u, "\beta_2", swap]
\end{tikzcd}
\rightsquigarrow
\begin{tikzcd}
 & j \arrow[d, swap]\\
 i \arrow[ur, swap] & t' \arrow[l, swap]  \arrow[d, swap,] & k \arrow[l, swap] \\
 & l \arrow[ur, swap] 
\end{tikzcd}.\]

Note that this is the same procedure specified in the definition of type \eqref{type1} in Section \ref{subsubsection::type1mutationdefinition}. Namely, we have the local configuration for mutation in Figure \ref{fig::quivermutation1} is the same as the figure above. Therefore, our type \eqref{type1} mutation generalizes classical quiver mutation.
\end{proof}

Another important characteristic of the mutation of quivers is its compatibility with flips on the associated triangulation.

\begin{theorem} \label{theorem::compatibilityofmutation}
Let $T$ be a triangulation of a marked surface (orientable or not) and let $(Q_T,P_T)$ be the associated partitioned quiver. Consider $\t$ an arc in $T$ and $t$ the associated vertex in $Q_T$. Then,
\[ \mu_t(Q,P) = \left(Q_{\mu_\t(T)}, P_{\mu_\t (T)}\right). \]
\end{theorem}

\begin{proof}
In order to prove the compatibility of partitioned quiver mutation and quasi-mutation on the level of the surface, we split the proof into cases depending on the type of mutation taking place. This is a direct consequence of the way in which we have defined our quiver mutation.

In the case of type \eqref{type1} mutation, note that this coincides with mutation of quivers in the classical case by Theorem \ref{theorem::classical} applying Proposition $4.8$ in \cite{FZ07}. Therefore, we only need to prove the statement holds in types \eqref{type2}, \eqref{type3} and \eqref{type4} mutations, which are mutations of quiver arising from quasi-cluster algebras.

Suppose that $\t$ is an arc of type \eqref{type2}. The type \eqref{type2} mutation rule for triangulations is shown in Figure \ref{Fig::Mutation2}. Then we associate a partitioned quiver $(Q_T,P_T)$ to the unmutated triangulation shown in Figure \ref{Fig::Mutation2}. After we perform mutation on $(Q_T,P_T)$, as defined in \ref{quiverMutation2} notice that the resulting quiver in section \ref{quiverMutation2} is associated with the mutated triangulation in Figure \ref{Fig::Mutation2}. Hence, in type \eqref{type2}, we have shown the compatibility of mutation.

If $\t$ is an arc of types \eqref{type3} or \eqref{type4}, the same argument as above applies looking at Figures \ref{Fig::Mutation3} and \ref{Fig::Mutation4} respectively and then verifying that the definitions in sections \ref{subsubsection::type3mutationdefinition} and \ref{subsubsection::type4mutationdefinition} respectively give the partitioned quivers associated with the mutated triangulations.
\end{proof}


\subsection{Exchange Relations} \label{subsec::exchangerlt}

In this subsection, mimicking quiver mutation from classical cluster algebras, we give a graphical interpretation of the mutation rules. Since we have four different types of mutation, we create a graphical exchange rule in each type. Assume that $T$ is a triangulation of a marked surface $(\S, \M)$, let $(Q_T, P_T)$ be its partitioned quiver and $\x \cup \y$ be the associated cluster, where $\x$ is the set of cluster variables and $\y$ is the set of coefficients.

We consider an arc $\t_t$ of $T$. Denote the variable in $\x$ associated to $t$ by $x_t$ and the vertex in $Q_T$ associated to $x_t$ by $t$. The flip of the arc $\t_t$ corresponds to the mutation in direction of $t$ of the seed $(\x \cup \y, Q_T, P_T)$. Here, we described the mutated set of cluster variables $\x'$ is given by $\x' = \left(\x \setminus x_t \right) \cup x_t'$ where $x_t'$ by different types of exchange relations. 

\subsubsection{Exchange Relation in Type \eqref{type1}}
Consider $\t_t$ to be an arc of type \eqref{type1}. Let $[\alpha_1 \alpha_2]$ and $[\beta_1 \beta_2]$ be $2$-paths through $t$ that are subpaths of $P_T$ and let $\{\g, \delta \}$ be the unique pair of arrows such that $\g$ and $\delta$ are concatenable in $Q_T$ with both $[\alpha_1 \alpha_2]$ and $[\beta_1 \beta_2]$, but not $P_T$-concatenable.
Moreover, let $i$ and $j$ be the endpoints of $[\alpha_1 \alpha_2]$ and $k$ and $l$ be the endpoints of of $[\beta_1 \beta_2]$.
Finally, suppose that $\delta$ is incident to both $j$ and $k$ and that $\g$ is incident to both $i$ and $l$. This situation is exactly the one illustrated by Remark \ref{rmk::Qlocal} \eqref{type1}.

Then, the mutated set of cluster variables $\x'$ is determined by the exchange relation:
\[x_t x_t'= x_i x_k + x_j x_l.\]
 
In other words, the product of the initial and the new quasi-cluster variables is obtained by taking the sum of product of the two pairs of variables associated with vertices that are both incident to $t$, but not incident one to the other.

\subsubsection{Exchange Relation in Type \eqref{type2}}
Consider $\t_t$ to be an arc of type \eqref{type2}. Let $i$ be the only vertex in $Q$ directly connected by arrow to $t$. This situation is exactly the one illustrated by Remark \ref{rmk::Qlocal} \eqref{type2}.

Then, the mutated set of cluster variables $\x'$ is determined by the exchange relation:
\[x_t x_t'= x_i.\]
In other words, the product of the initial and the new quasi-cluster variables is equal to the variable associated with the only vertex in $Q_T$ connected by a single arrow to $t$.
%


\subsubsection{Exchange Relation in Type \eqref{type3}}
Consider $\t_t$ to be an arc of type \eqref{type3}. Let $x$ be the only vertex in $Q$ directly connected by a $2$-cycle to $t$. Let $i$ and $k$ be the two other vertex connected by an arrow to $t$. This situation is exactly the one illustrated by Remark \ref{rmk::Qlocal} \eqref{type3}.

The graphical interpretation for type \ref{type4} mutation is more complicated because we have to take into account both the local configuration of the mutated vertex $t$ before and after the mutation. However, taking into account the data of the partition, we can think of the exchange relation as the product of all the edges adjacent to $t$ that have changed parts of the partition after mutation, plus the sum of the two edges that changes parts of the partition and direction after the mutation all squared.  

In Figure \ref{fig::quivermutation3}, given the local picture of the mutated vertex $t$ before and after mutation, the pairs of edges $\{\alpha_3, \beta_3\}$, $\{\alpha_2, \beta_2\}$, $\{\alpha_1, \beta_1\}$, and $\{\alpha_4, \beta_4\}$ are not in the same path, which constitutes the term $x_j^2x_ix_k$. Moreover, edges $\alpha_1$ and $\alpha_4$ will change direction due to rule (3) in \ref{subsubsection::type4mutationdefinition}, which constitutes the term $(x_i+x_k)^2$. 

Then, the mutated set of cluster variables $\x'$ is determined by the exchange relation:
\[x_t x_t'= (x_i + x_k)^2 + x_i x_j^2 x_k.\]

\subsubsection{Exchange Relation in Type \eqref{type4}}
Consider $\t_t$ to be an arc of type \eqref{type4}. Let $i$ be the only vertex in $Q$ directly connected by arrow to $t$. This situation is exactly the one illustrated by Remark \ref{rmk::Qlocal} \eqref{type4} and is really similar to the Type \eqref{type2} situation.

Then, the mutated set of cluster variables $\x'$ is determined by the exchange relation:
\[x_t x_t'= x_i.\]
In other words, the product of the initial and the new quasi-cluster variables is equal to the variable associated with the only vertex in $Q_T$ connected by a single arrow to $t$.

\subsubsection{Compatibility of Quasi-Cluster Variables}
With the construction of our partitioned quiver, we show that we obtain the same quasi-cluster algebra defined in terms of quasi-triangulations of surfaces. The key observation to obtain that statement is that, in each type of mutation of the quiver and in each exchange relation, we do not consider the direction of the arrows, but rather subpaths in $P_T$. Changing the orientation of a boundary components reverse the orientation of all paths in $P_T$ associated to marked points on this boundary components. Since the partition of the arrows of $Q_T$ in paths $P_T$ stays the same, that does not the exchange relations and therefore, the cluster variables obtained.


\begin{crl}\label{theorem::isomorphicalgebras}
The quasi-cluster variables obtained via the surface construction by \cite{DP15} as in Definition \ref{def::quasimutation} are in bijection with the quasi-cluster variables obtained by the partitioned quiver mutation. That is, the quasi cluster algebras obtained are isomorphic. 
\end{crl}

\begin{proof}
It suffices to show that the quasi-cluster variables obtained by partitioned quiver mutation are the same as the quasi-cluster variables obtained by \cite{DP15}. In other words, it suffices to show that the exchange relations defined in Section \ref{subsec::exchangerlt} are the same exchange relations that \cite{DP15} define. By definition of the exchange relations in Section \ref{subsec::exchangerlt}, we have that the quasi-cluster variables given in \cite{DP15} are the exact same as the quasi-cluster variables given our exchange relations. 
\end{proof}

\subsection{Other Quivers from Non-Orientable Surfaces}  

We present in this section different quivers arising from non-orientable surfaces. Firstly, we associate quiver directly from a triangulation of a non-orientable surface, but with a slightly different definition of compatibility of arcs, as used by Wilson, \cite{Wil18, Wil19, Wil20}. Secondly, we associate a quiver from a triangulation of a non-orientable surface via an orientable double cover of this surface.

\subsubsection{Wilson-Triangulations}

Our goal to adapt our definition of quivers from triangulations of non-orientable surfaces to the triangulations as defined by J. Wilson, \cite{Wil20}. Let's introduce the definition of arcs and compatibility according to this author. The purpose of these new definitions was to obtain only irreducible polynomials in the exchange relations.

An arc is \dfn{Wilson-admissible} if it is not a type \eqref{type3} arc, that is, it does not bound a Möbius strip with one point on the boundary. We write in short a \dfn{W-arc} for a Wilson admissible arc.
Two arcs $\t$ and $\t'$ in a triangulation are \dfn{Wilson-compatible} or \dfn{W-compatible} if they do not intersect or if one is a regular arc of type \eqref{type2} and the other is quasi-arc going through the same crosscap as the first one (see \autoref{subfig::warc3} and \autoref{subfig::warc4}).

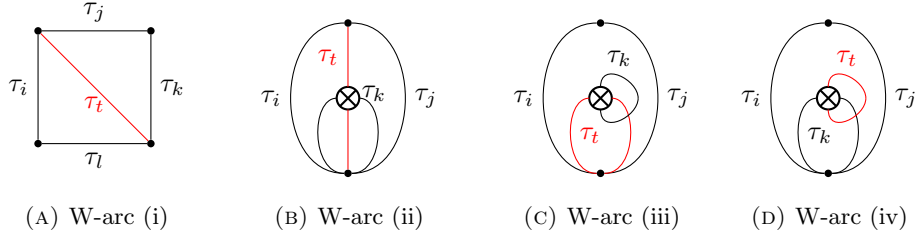
\begin{figure}[htb]
    \centering
    \begin{subfigure}[b]{0.23\textwidth}
        \centering
            \[ \begin{tikzpicture}[scale = 1.5]
             \tikzstyle{p}=[circle,fill, scale=0.3]
            \node[p] (1) at (0, 0) {};
            \node[p] (2) at (0, 1) {};
            \node[p] (3) at (1, 1) {};
            \node[p] (4) at (1, 0) {};
            \draw (1) -- node [left] {$\t_i$} (2);
            \draw (2) -- node [above] {$\t_j$} (3);
            \draw (3) -- node [right] {$\t_k$} (4);
            \draw (4) -- node[below] {$\t_l$} (1);
            \draw[red] (2) -- node [left, below] {$\t_t$} (4);
            \end{tikzpicture} \]
        \caption{W-arc \eqref{wtype1}}
        \label{subfig::warc1}
     \end{subfigure}
     \hfill
     \begin{subfigure}[b]{0.23\textwidth}
        \centering
            \[ \begin{tikzpicture}
            \tikzstyle{p}=[circle,fill, scale=0.3]
            \node[p] (1) at (0, -1) {};
            \node[p] (2) at (0, 1) {};
            \draw[thick] (0, 0) circle (0.15cm);
            \draw[rotate=45, thick] (0,-0.15) -- (0,0.15);
            \draw[rotate=45, thick] (-0.15,0) -- (0.15, 0);
            \draw (1) .. controls (-1,-1) and (-1,1) ..  node [left] {$\t_i$} (2);
            \draw (1) .. controls (1,-1) and (1,1) .. node [right] {$\t_j$} (2);
            \draw (1) .. controls (-0.5,-1) and (-0.5,0) .. (-0.15,0);
            \draw (1) .. controls (0.5,-1) and (0.5,0) .. node [above, near end] {$\t_k$} (0.15,0);
            \draw[red] (1) -- (0,-0.15);
            \draw[red] (0,0.15) -- node [left] {$\t_t$} (2);
            \end{tikzpicture}\]
        \caption{W-arc \eqref{wtype2}}
        \label{subfig::warc2}
     \end{subfigure}
     \hfill
    \begin{subfigure}[b]{0.23\textwidth}
        \centering
            \[ \begin{tikzpicture}
            \tikzstyle{p}=[circle,fill, scale=0.3]
            \node[p] (1) at (0, -1) {};
            \node[p] (2) at (0, 1) {};
            \draw[thick] (0, 0) circle (0.15cm);
            \draw[rotate=45, thick] (0,-0.15) -- (0,0.15);
            \draw[rotate=45, thick] (-0.15,0) -- (0.15, 0);
            \draw (1) .. controls (-1,-1) and (-1,1) ..  node [left] {$\t_i$} (2);
            \draw (1) .. controls (1,-1) and (1,1) .. node [right] {$\t_j$} (2);
            \draw[red] (1) .. controls (-0.5,-1) and (-0.5,0) .. node [right] {$\t_t$} (-0.15,0);
            \draw[red] (1) .. controls (0.5,-1) and (0.5,0) .. (0.15,0);
            \draw (0,-0.15) .. controls (0,-0.5) and (0.5,-0.25) .. (0.5,0);
            \draw (0,0.15) .. controls (0,0.5) and (0.5,0.25) .. node [above] {$\t_k$}(0.5,0);
            \end{tikzpicture}\]
        \caption{W-arc \eqref{wtype3}}
        \label{subfig::warc3}
     \end{subfigure}
     \begin{subfigure}[b]{0.23\textwidth}
        \centering
            \[ \begin{tikzpicture}
            \tikzstyle{p}=[circle,fill, scale=0.3]
            \node[p] (1) at (0, -1) {};
            \node[p] (2) at (0, 1) {};
            \draw[thick] (0, 0) circle (0.15cm);
            \draw[rotate=45, thick] (0,-0.15) -- (0,0.15);
            \draw[rotate=45, thick] (-0.15,0) -- (0.15, 0);
            \draw (1) .. controls (-1,-1) and (-1,1) ..  node [left] {$\t_i$} (2);
            \draw (1) .. controls (1,-1) and (1,1) .. node [right] {$\t_j$} (2);
            \draw (1) .. controls (-0.5,-1) and (-0.5,0) .. node [right] {$\t_k$} (-0.15,0);
            \draw (1) .. controls (0.5,-1) and (0.5,0) .. (0.15,0);
            \draw[red] (0,-0.15) .. controls (0,-0.5) and (0.5,-0.25) .. (0.5,0);
            \draw[red] (0,0.15) .. controls (0,0.5) and (0.5,0.25) .. node [above] {$\t_t$}(0.5,0);
            \end{tikzpicture}\]
        \caption{W-arc \eqref{wtype4}}
        \label{subfig::warc4}
     \end{subfigure}
    \caption{Four types of W-arcs}
    \label{fig::4typeswarcs}
\end{figure}

\begin{remark} \label{rmk::WQlocal}
Therefore, the W-arcs are split in four cases as well, but slightly different than the ones studied before; these four cases are illustrated at \autoref{fig::4typeswarcs}. An W-arc $\t$ is either: \begin{enumerate} [(i)]
    \item \label{wtype1} an arc cutting two different regular triangles such that $\t_i \neq \t_j$ and $\t_k \neq \t_l$ (see \autoref{subfig::warc1});
    
    \item \label{wtype2} an arc $\t_t$ cutting two different regular triangles sharing a common edge (see \autoref{subfig::warc2});
    
    \item \label{wtype3} an arc $\t_t$ cutting a quasi-arc (see \autoref{subfig::warc3});
    
    \item \label{wtype4} a quasi-arc cutting a type \eqref{type2} arc (see \autoref{subfig::warc4});.
\end{enumerate}
\end{remark}

We will refer to these four types of $W$-arcs as $W$-types \ref{wtype1}, \ref{wtype2}, \ref{wtype3} and \ref{wtype4}. A \dfn{Wilson-admissible triangulation} or shortly a \dfn{W-triangulation} is a maximal collection of Wilson-admissible triangles. In this context, we keep the property that in any W-triangulation, any W-arc can be uniquely flipped to obtain a new W-arc and a new W-triangulation. Figure \ref{Fig::W-arcs} shows the mutation of each W-type of arcs.
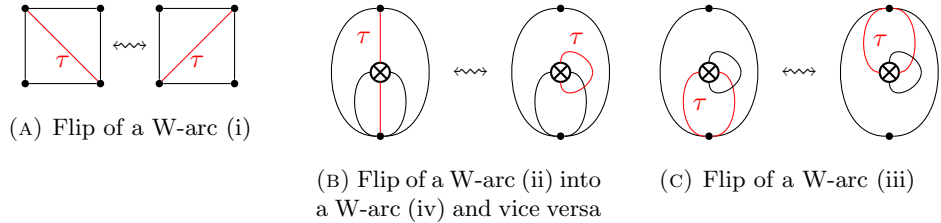
\begin{figure}
    \begin{subfigure}[t]{0.3\textwidth}
        \centering
            \[ \begin{tikzpicture}[scale = 1, baseline={([yshift=-.5ex]current bounding box.center)}]
             \tikzstyle{p}=[circle,fill, scale=0.3,  ]
            \node[p] (1) at (0, 0) {};
            \node[p] (2) at (0, 1) {};
            \node[p] (3) at (1, 1) {};
            \node[p] (4) at (1, 0) {};
            \draw (1) -- (2);
            \draw (2) -- (3);
            \draw (3) -- (4);
            \draw (4) -- (1);
            \draw[red] (2) -- node [left, below] {$\t$} (4);
            \end{tikzpicture}
            \leftrightsquigarrow
            \begin{tikzpicture}[scale = 1,  baseline={([yshift=-.5ex]current bounding box.center)}]
             \tikzstyle{p}=[circle,fill, scale=0.3]
            \node[p] (1) at (0, 0) {};
            \node[p] (2) at (0, 1) {};
            \node[p] (3) at (1, 1) {};
            \node[p] (4) at (1, 0) {};
            \draw (1) -- (2);
            \draw (2) -- (3);
            \draw (3) -- (4);
            \draw (4) -- (1);
            \draw[red] (1) -- node [right, below] {$\t$} (3);
            \end{tikzpicture}\]
        \caption{Flip of a W-arc  \eqref{wtype1}}
        \label{subfig::mutwarc1}
     \end{subfigure}
     \hfill
     \begin{subfigure}[t]{0.3\textwidth}
        \centering
            \[ \begin{tikzpicture}[baseline={([yshift=-.5ex]current bounding box.center)}, scale=0.85]
            \tikzstyle{p}=[circle,fill, scale=0.3]
            \node[p] (1) at (0, -1) {};
            \node[p] (2) at (0, 1) {};
            \draw[thick] (0, 0) circle (0.15cm);
            \draw[rotate=45, thick] (0,-0.15) -- (0,0.15);
            \draw[rotate=45, thick] (-0.15,0) -- (0.15, 0);
            \draw (1) .. controls (-1,-1) and (-1,1) ..  (2);
            \draw (1) .. controls (1,-1) and (1,1) ..  (2);
            \draw (1) .. controls (-0.5,-1) and (-0.5,0) .. (-0.15,0);
            \draw (1) .. controls (0.5,-1) and (0.5,0) .. (0.15,0);
            \draw[red] (1) -- (0,-0.15);
            \draw[red] (0,0.15) -- node [left] {$\t$} (2);
            \end{tikzpicture}
            \leftrightsquigarrow
            \begin{tikzpicture}[baseline={([yshift=-.5ex]current bounding box.center)}, scale=0.85]
            \tikzstyle{p}=[circle,fill, scale=0.3]
            \node[p] (1) at (0, -1) {};
            \node[p] (2) at (0, 1) {};
            \draw[thick] (0, 0) circle (0.15cm);
            \draw[rotate=45, thick] (0,-0.15) -- (0,0.15);
            \draw[rotate=45, thick] (-0.15,0) -- (0.15, 0);
            \draw (1) .. controls (-1,-1) and (-1,1) ..  (2);
            \draw (1) .. controls (1,-1) and (1,1) .. (2);
            \draw (1) .. controls (-0.5,-1) and (-0.5,0) .. (-0.15,0);
            \draw (1) .. controls (0.5,-1) and (0.5,0) .. (0.15,0);
            \draw[red] (0,-0.15) .. controls (0,-0.5) and (0.5,-0.25) .. (0.5,0);
            \draw[red] (0,0.15) .. controls (0,0.5) and (0.5,0.25) .. node [above] {$\t$}(0.5,0);
            \end{tikzpicture}\]
        \caption{Flip of a W-arc \eqref{wtype2} into a W-arc \eqref{wtype4} and vice versa}
        \label{subfig::mutationwarc24}
     \end{subfigure}
     \hfill
\begin{subfigure}[t]{0.3\textwidth}
        \centering
            \[ \begin{tikzpicture}[baseline={([yshift=-.5ex]current bounding box.center)}, scale=0.85]
            \tikzstyle{p}=[circle,fill, scale=0.3]
            \node[p] (1) at (0, -1) {};
            \node[p] (2) at (0, 1) {};
            \draw[thick] (0, 0) circle (0.15cm);
            \draw[rotate=45, thick] (0,-0.15) -- (0,0.15);
            \draw[rotate=45, thick] (-0.15,0) -- (0.15, 0);
            \draw (1) .. controls (-1,-1) and (-1,1) ..  (2);
            \draw (1) .. controls (1,-1) and (1,1) .. (2);
            \draw[red] (1) .. controls (-0.5,-1) and (-0.5,0) .. node [right] {$\t$} (-0.15,0);
            \draw[red] (1) .. controls (0.5,-1) and (0.5,0) .. (0.15,0);
            \draw (0,-0.15) .. controls (0,-0.5) and (0.5,-0.25) .. (0.5,0);
            \draw (0,0.15) .. controls (0,0.5) and (0.5,0.25) .. (0.5,0);
            \end{tikzpicture}
            \leftrightsquigarrow
            \begin{tikzpicture}[baseline={([yshift=-.5ex]current bounding box.center)}, scale=0.85]
            \tikzstyle{p}=[circle,fill, scale=0.3]
            \node[p] (1) at (0, -1) {};
            \node[p] (2) at (0, 1) {};
            \draw[thick] (0, 0) circle (0.15cm);
            \draw[rotate=45, thick] (0,-0.15) -- (0,0.15);
            \draw[rotate=45, thick] (-0.15,0) -- (0.15, 0);
            \draw (1) .. controls (-1,-1) and (-1,1) ..  (2);
            \draw (1) .. controls (1,-1) and (1,1) .. (2);
            \draw[red] (2) .. controls (-0.5,1) and (-0.5,0) .. node [right] {$\t$} (-0.15,0);
            \draw[red] (2) .. controls (0.5,1) and (0.5,0) .. (0.15,0);
            \draw (0,-0.15) .. controls (0,-0.5) and (0.5,-0.25) .. (0.5,0);
            \draw (0,0.15) .. controls (0,0.5) and (0.5,0.25) .. (0.5,0);
            \end{tikzpicture}\]
        \caption{Flip of a W-arc \eqref{wtype3}}
        \label{subfig::mutationwarc3}
     \end{subfigure}
     \hfill
    \caption{Mutation of W-arcs}
    \label{Fig::W-arcs}
\end{figure}

Besides, just as explained in Section \ref{section::ourquiver}, we can define a quiver $Q_T$ and a set of paths $P_T$ associated with a W-triangulation $T$. Once again, we fix an orientation for each of the boundary component. Consider two arcs $\t_1$ and $\t_2$ sharing a common endpoint $i$ in $\M$ such there is no other W-arc attached to $i$ between $\t_1$ and $\t_2$. Suppose moreover that $\t_2$ follows $\t_1$ according to the orientation of the boundary on which lies $i$. Either none of $\t_1$ and $\t_2$ intersects with another W-arc, either one of them does.
In both cases, there is an arrow $1 \rightarrow 2$ in $Q_T$. 
Moreover, if $\t_1$ intersect a W-arc $\t_3$, add an arrow $3 \rightarrow 1$.
Similarly, if $\t_2$ intersect a W-arc $\t_3$, add an arrow $2 \rightarrow 3$.

As with our previous construction, marked points in $M$ give rise to paths in $Q_T$; we denote the set of all these paths by $P_T$.

We can also obtain $Q_T$ by constructing it locally around a vertex $t$ associated to an arc $\t_t$ as follows.

\begin{enumerate}[(i)]
    \item The W-type \eqref{wtype1} is a special case of the ordinary type \eqref{type1} with $a \neq c$ and $b \neq d$. We therefore refer the reader to the Remark \ref{rmk::Qlocal} \eqref{type1}.
    
    \item If $\t_t$ is of W-type \eqref{wtype2}, then $Q_T$ contains the subgraph shown in \autoref{subfig::mutwvertex24} on left, with $[\alpha_1 \alpha_2 \alpha_3 \alpha_4]$ and $[\beta_1 \beta_2]$ subpaths in $P_T$. Here,$i$, $j$ and $k$ are respectively associated to the arcs $\t_i$, $\t_j$ and $\t_k$ from \autoref{subfig::warc2}. When we mutate the quiver at the vertex $t$, it changes locally as illustrated in \autoref{subfig::mutwarc1} on right. 
    
    Denote $p_b$ the path in $P_T$ in which lies $[\beta_1 \beta_2]$.
    The mutated partition of $Q_1$ $\mu_t(P_T)$ is obtained from $P_T$ by removing the subpath $[\beta_1 \beta_2]$ from $p_b$ and replace it by $\beta$ where $\beta$ is such that $s(\beta) = s([\beta_1 \beta_2])$ and $t(\beta) = t([\beta_1 \beta_2])$ respectively.
    
    \item  If $\t_t$ is of W-type \eqref{wtype3}, then $Q_T$ contains the subgraph shown in \autoref{subfig::mutwvertex3} on left, with $[\alpha_1 \alpha_2 \alpha_3 \alpha_4]$ and $[\beta_1 \beta_2]$ subpaths in $P_T$. Here,$i$, $j$ and $k$ are respectively associated to the arcs $\t_i$, $\t_j$ and $\t_k$ from \autoref{subfig::warc3}. When we mutate the quiver at the vertex $t$, it changes locally as illustrated in \autoref{subfig::mutwvertex3} on right.
    
    Denote $p_a$ and $p_b$ the paths in $P_T$ in which lies $[\alpha_1 \alpha_2 \alpha_3 \alpha_4]$ and $\beta$ respectively.
    The mutated partition of $Q_1$ $\mu_t(P_T)$ is obtained from $P_T$ by performing the following steps: \begin{enumerate}
        \item remove the arrows $\alpha_1$, $\alpha_2$, $\alpha_3$, $\alpha_4]$ and $\beta$; 
        \item add an arrow $s([\alpha_1 \alpha_2 \alpha_3 \alpha_4]) \xrightarrow{\alpha} \t([\alpha_1 \alpha_2 \alpha_3 \alpha_4])$
        \item add a subpath $[\beta_1 \beta_2 \beta_3 \beta_4]$ it by $[\beta_1 \beta_2 \beta_3 \beta_4]$.
    \end{enumerate}
    
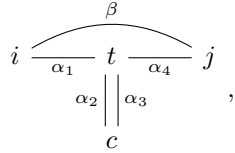
\[ \begin{tikzcd}
     i  \arrow[r, dash, "\alpha_1", swap] \arrow[rr, dash, "\beta", bend left] & t \arrow[r, dash, "\alpha_4", swap] \arrow[d, shift right, dash, swap, "\alpha_2"] \arrow[d, shift left, dash, "\alpha_3"]& j \\
     & c
    \end{tikzcd},\]
    with $[\alpha_1 \alpha_2 \alpha_3 \alpha_4]$ and $[\beta]$ subpaths in $P_T$. Here,$i$, $j$ and $k$ are respectively associated to the arcs $\t_i$, $\t_j$ and $\t_k$ from Remark \ref{rmk::WQlocal} \eqref{wtype3}.
    
    \item  If $\t_t$ is of W-type \eqref{wtype4}, then $Q_T$ contains the following subgraph:
    
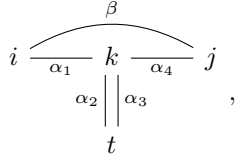
\[ \begin{tikzcd}
     i  \arrow[r, dash, "\alpha_1", swap] \arrow[rr, dash, "\beta", bend left] & k \arrow[r, dash, "\alpha_4", swap] \arrow[d, shift right, dash, swap, "\alpha_2"] \arrow[d, shift left, dash, "\alpha_3"]& j \\
     & t
    \end{tikzcd},\]
    with $[\alpha_1 \alpha_2 \alpha_3 \alpha_4]$ and $[\beta]$ subpaths in $P_T$. Here,$i$, $j$ and $k$ are respectively associated to the arcs $\t_i$, $\t_j$ and $\t_k$ from Remark \ref{rmk::WQlocal} \eqref{wtype3}.
    Observe that for this last local configuration, we did not uniquely consider the vertices adjacent to $t$, but also some vertices adjacent to them. We indeed need them to define the mutation of the quiver in this type of arc.
\end{enumerate}

\autoref{fig::wmutation} illustrated locally the mutation W-arcs of type \eqref{wtype2}, \eqref{wtype3} and \eqref{wtype4} (the mutation of a W-arc of type \eqref{wtype1} is defined exactly as the mutatuion of a W-arc of type \eqref{type1}).

\begin{figure}[h!]
    \centering
    \begin{subfigure}[b]{0.75\textwidth}
        \centering
           \[ \begin{tikzcd}
             i  \arrow[r, dash, "\alpha_1", NavyBlue] \arrow[dr, dash, "\beta_1", swap, BurntOrange] & k \arrow[r, dash, "\alpha_4", NavyBlue] \arrow[d, shift right, dash, swap, "\alpha_2", NavyBlue] \arrow[d, shift left, dash, "\alpha_3", NavyBlue]& j \arrow[dl, dash, "\beta_2", BurntOrange]  \\
             & \textcolor{red}{t}
            \end{tikzcd}
            \leftrightsquigarrow
            \begin{tikzcd}
             i  \arrow[r, dash, "\alpha_1", NavyBlue, swap] \arrow[rr,  "\beta", BurntOrange, bend left, dash] & k \arrow[r, dash, "\alpha_4", NavyBlue, swap] \arrow[d, shift right, dash, swap, "\alpha_2", NavyBlue] \arrow[d, shift left, dash, "\alpha_3", NavyBlue]& j   \\
             & \textcolor{red}{t}
          \end{tikzcd} \] 
        \caption{Vertex associated with a W-arc \eqref{wtype2} on left and vertex associated with a W-arc \eqref{wtype4} on right}
        \label{subfig::mutwvertex24}
     \end{subfigure}
     \hfill
     
     \begin{subfigure}[b]{1\textwidth}
        \centering
           \[ \begin{tikzcd}
             i  \arrow[r, dash, "\alpha_1", NavyBlue, swap] \arrow[rr,  "\beta", BurntOrange, bend left, dash] & \textcolor{red}{t} \arrow[r, dash, "\alpha_4", NavyBlue, swap] \arrow[d, shift right, dash, swap, "\alpha_2", NavyBlue] \arrow[d, shift left, dash, "\alpha_3", NavyBlue]& j   \\
             & c
            \end{tikzcd}
            \leftrightsquigarrow
            \begin{tikzcd}
             i  \arrow[r, dash, "\beta_1", BurntOrange, swap] \arrow[rr,  "\alpha", NavyBlue, bend left, dash] & \textcolor{red}{t} \arrow[r, dash, "\beta_4", BurntOrange, swap] \arrow[d, shift right, dash, swap, "\beta_2", BurntOrange] \arrow[d, shift left, dash, "\beta_3", BurntOrange]& j   \\
             & c
          \end{tikzcd} \] 
        \caption{Vertex associated with a W-arc \eqref{wtype3}}
        \label{subfig::mutwvertex3}
     \end{subfigure}
     \hfill
    \caption{Mutation of quivers from W-triangulation for vertex associated with W-arcs of type \eqref{wtype2}, \eqref{wtype3} and \eqref{wtype4}.}
    \label{fig::wmutation}
\end{figure}
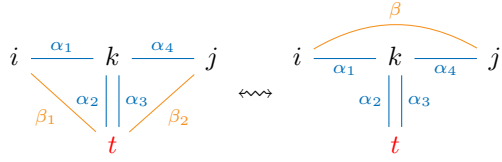
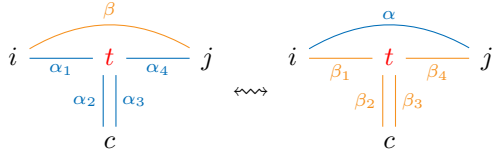


\subsubsection{Double Cover}

In \cite{DP15}, the authors lift certain quasi-triangulations to the associated surface's orientable double cover. They showed that the cluster algebra associated with the this orientable double cover, in the sense of \cite{FST08}, coincides with their definition of quasi-cluster algebra in most cases. Namely, this lift to the double cover completely coincides with the work of \cite{FST08} except in the case of where the quasi-triangulation contains quasi-arcs.

In this section, we will consider the case where our quasi-triangulations do not contain any quasi-arcs. From this, we will define a \emph{double quiver} associated to the double cover of our surface $\SM$. In order to this, we review some background information about the double cover, following Section 6 of \cite{DP15}.

Let $\SM$ be a non-orientable marked surface and let $T$ be a quasi-triangulation that does not contain a quasi-arc. Recall that $\SM$ admits a minimal orientable 2:1 cover called its double cover $(\mathbf{\Bar{S}, \Bar{M}})$. 
This surface $(\mathbf{\Bar{S}, \Bar{M}})$ can be endowed with a free $\mathbb{Z}_2$ action such that $(\mathbf{\Bar{S}, \Bar{M}})/\mathbb{Z}_2 \simeq \SM$. 
Say $\mathbb{Z}_2 = \{1,x\}$, then each arc $\t$ in the marked surface $\SM$ lifts to two compatible arcs $1\overline{\t}$ and $x \overline{\t}$ in $(\mathbf{\Bar{S}, \Bar{M}})$. 
Using this, we have that $T = \{\t_1, \dots, \t_n\}$ lifts to a triangulation $\bar{T} = \{1\bar{\t_1}, \dots, x\bar{\t_n},1\bar{\t_1}, \dots, x\bar{\t_n}\}$ of $(\mathbf{\Bar{S}, \Bar{M}})$ that is invariant under the $\mathbb{Z}_2$-action.

Using this lift of a triangulation $T$ to $\bar{T}$, we can associate a double quiver to the triangulation $\bar{T}$. By Theorem \ref{theorem::classical}, we have that our partitioned quiver coincides with the quiver defined by \cite{FST08} when our surface is orientable. Since $\bar{T}$ is a triangulation of an orientable marked surface $(\mathbf{\Bar{S}, \Bar{M}})$, the partitioned quiver we obtain from $\bar{T}$ is the same as the quiver defined by \cite{FST08}.

\begin{example}
Consider the quasi-triangulation $T$ of the M\"obius strip with 3 marked points shown in Figure \ref{fig::doubleCover}. The lift of this triangulation is shown on its double cover: the annulus. Moreover, the corresponding partitioned quiver as well as the double quiver that we associate to $\bar{T}$ are both shown in Figure \ref{fig::DoubleQuiver}.
\end{example}

\begin{remark}
We want to emphasize that the construction of the double quiver that aligns with the work of \cite{FST08} only works for triangulations without quasi-arcs. To highlight why this construction fails when the triangulation contains a quasi-arc, notice that if we take a triangulation of the M\"obius strip with 3 marked points that contains a quasi-arc, the quasi-arc lifts to a non-contractible loop in $(\mathbf{\Bar{S}, \Bar{M}})$ that cannot be a part of any triangulation. 
\end{remark}


\begin{figure}[h]
 \centering
 \begin{tikzpicture} [baseline={([yshift=-.5ex]current bounding box.center)}, scale=0.4]
    \draw[very thick] (0,0) circle (3cm)
    node[above] at (0,-3){$i$}
    node[right] at (-3,0){$j$}
    node[left] at (3,0){$k$};
    
    \draw[thick] (0, 0) circle (0.25cm);
    \draw[rotate=45] (0,-0.25) -- (0,0.25);
    \draw[rotate=45] (-0.25,0) -- (0.25, 0);
    
    \tikzstyle{p}=[circle,fill, scale=0.3]
    \node[p] (1) at (0, 3) {};
    \node[p] (2) at (-2,-2.23) {};
    \node[p] (3) at (2,-2.23){};
    
    \draw [gray, latex-](0.5,3.5) -- (-0.5,3.5);
    
    \draw [red](1).. controls (1,1) and (0.8, 0).. node[midway,right] {$1$} (0.25,0);
    \draw [red] (1) .. controls (-1,1) and (-0.8,0) .. (-0.25,0);
    
    \draw [blue] (1) .. controls (0.3,1) .. (0.1,0.25);
    \draw [blue] (-0.1,-0.25) .. controls (-1,-1.8) .. node[midway,left] {$2$}(2);
    
    \draw[pink] (1) .. controls (-0.3,1) .. (-0.1,0.25);
    \draw[pink] (0.1,-0.25) .. controls (1,-1.8)  ..node[midway,right] {$3$} (3);
    \draw[latex-](3.5,0) -- node[midway,above] {$2:1$}(5.5,0);
\end{tikzpicture}
\begin{tikzpicture}[baseline={([yshift=-.5ex]current bounding box.center)}, scale=0.5]
    \tikzstyle{p}=[circle,fill, scale=0.3]
    \node[p] (1) at (1, 0) {};
    \node[p] (2) at (2, 0) {};
    \node[p] (3) at (4, 0) {};
    \node[p] (4) at (6, 0) {};
    \node[p] (5) at (7, 0) {};
    \node[p] (6) at (9, 0) {};
    \node[p] (7) at (1.5, -3) {};
    \node[p] (8) at (3.5, -3) {};
    \node[p] (9) at (4.5, -3) {};
    \node[p] (10) at (6.5, -3) {};
    \node[p] (11) at (8.5, -3) {};
    \node[p] (12) at (9.5, -3) {};
   \draw (0,0) -- (10,0);
   \draw [gray, -latex](4.7,0.3) -- (5.2,0.3);
   \draw (0,-3) -- (10, -3);
   \draw [gray, latex-](4.7,-3.3) -- (5.2,-3.3);
   \draw[blue] (1) -- (7);
   \draw[pink] (2) -- (7);
   \draw[red] (7) -- node[midway,right]{$1'$} (3);
   \draw[blue] (8) -- node[midway,right]{$3$} (3);
   \draw[pink] (9) -- node[midway,right]{$2$}(3);
   \draw[red] (10) -- node[midway,right]{$1$} (3);
   \draw[blue] (10) -- node[midway,right]{$3'$} (4);
   \draw[pink] (10) --node[midway,right]{$2'$} (5);
   \draw[red] (10) --node[midway,right]{$1'$} (6);
   \draw[blue] (11) -- (6);
   \draw[pink] (12) -- (6);
   \draw[red] (0,-2.5) -- (7);
   \draw[red] (10,-0.5) -- (6);
    \draw[gray, dotted] (0,0) -- (0,-3);
    \draw[gray, dotted] (3,0) -- (3,-3);
    \draw[gray, dotted] (5,0) -- (5,-3);
    \draw[gray, dotted] (8,0) -- (8,-3);
    \draw[gray, dotted] (10,0) -- (10,-3);
\end{tikzpicture}
\caption{Lift of the triangulation to its orientable double cover.}
\label{fig::doubleCover}
\end{figure}
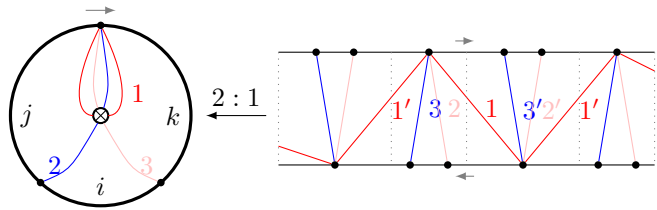

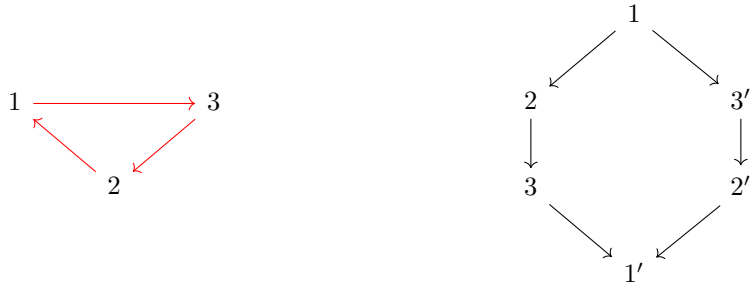
\begin{figure}
\begin{subfigure}{0.45\textwidth}
   \[ 
\begin{tikzcd}
 1 \arrow[rr, red]& & 3\arrow[dl,red] \\
&2 \arrow [ul,red]\\
\end{tikzcd}\]
\end{subfigure}
\hfill
\begin{subfigure}{0.45\textwidth}
   \[ 
\begin{tikzcd}
&1\arrow[dr]\arrow[dl] & \\
2\arrow[d]&& 3' \arrow [d]\\
3\arrow[dr]&& 2' \arrow [dl]\\
&1' & \\
\end{tikzcd}\]
\end{subfigure}
\caption{The corresponding partitioned quiver (left) and double quiver (right) associated to $\bar{T}$.}
\label{fig::DoubleQuiver}
\end{figure}

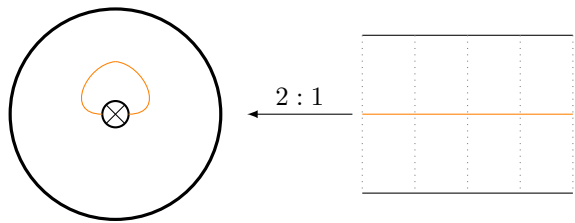
\begin{figure}
 \centering
 \begin{tikzpicture}[baseline={([yshift=-.5ex]current bounding box.center)}, scale=0.7]
    \draw[very thick] (0,0) circle (2cm);
  
    \draw[thick] (0, 0) circle (0.25cm);
    \draw[rotate=45] (0,-0.25) -- (0,0.25);
    \draw[rotate=45] (-0.25,0) -- (0.25, 0);

    \draw[orange] (0.25,0) ..  controls (1,0) and (0.55,0.95) ..  (0,1);
    \draw[orange] (0,1) .. controls (-0.55,0.95) and (-1,0)  .. (-0.25,0);
    
    \draw[latex-] (2.5,0) -- node[midway,above] {$2:1$} (4.5,0);
 \end{tikzpicture}
 \begin{tikzpicture}[baseline={([yshift=-.5ex]current bounding box.center)}, scale=0.7]
  \draw (0,0) -- (4,0);
   \draw (0,-3) -- (4, -3);
   \draw[orange] (0,-1.5) -- (4, -1.5);
   \draw[gray,dotted] (0,0) -- (0, -3);
   \draw[gray,dotted] (1,0) -- (1, -3);
   \draw[gray,dotted] (2,0) -- (2, -3);
   \draw[gray,dotted] (3,0) -- (3, -3);
   \draw[gray,dotted] (4,0) -- (4, -3);

\end{tikzpicture}
 \caption{Lift of a quasi-arc in the M\"obius strip to the annulus.}
 \label{fig::liftQuasiArc}
 \end{figure}

\section{Unistructiality of the M\"obius Strip} \label{section::unistructurality}

In this section, we divulge from our discussion of partitioned quiver to prove the unistructurality of the M\"obius strip.

\begin{definition}
Let $\A(X,T)$ be a cluster algebra, where $X$ is a cluster and $T$ is a triangulation of a given surface $\SM$ and denote the set of cluster variables of $\A(X,T)$ by $\X$.
Suppose there exists a surface $(\S^*,\M^*)$, a triangulation $T^*$ of this surface and a cluster $X^*$ such that the seed $(X^*,T^*)$ generates by mutations the exact same cluster variables as $(X,T)$, that is $\X$.
Then, $\A(X,T)$ is said to be unistructural if $\SM = (\S^*, \M^*)$.

\end{definition}

Since we are going to prove the unistructurality of the M\"obius strip, the only finite type of quasi-cluster algebras, we will introduce some theorems about finite types of cluster algebras.

\begin{theorem}\cite{FZ03}\label{thm::FominFinite}
A finite cluster algebra arising from a surface $\A(X,T)$ is mutational equivalent to a polygon, a polygon with a puncture, or a M\"obius Strip.
\end{theorem}


\begin{theorem} \label{theorem::finitetype} \cite{DP15}
A quasi-cluster algebra is of finite type if and only if it's arising from a M\"obius strip with at least one marked point on the boundary or it is a cluster algebra of finite type.
\end{theorem}

Using these finite type classifications, we show unistructurality for the only non-orientable finite type surface. 

\begin{theorem}
The quasi-cluster algebra arising from a M\"obius strip is unistructural. 
\end{theorem}

\begin{proof}
Let $T$ be a triangulation of a M\"obius strip with $n$ marked points on the boundary and consider the set $\X$ of cluster of cluster variables of the cluster algebra $\A(X,T)$ arising from that triangulation. Suppose there is another cluster algebra $\A(Y,X^*,T^*)$ that generates the exact same set $\X$ of cluster variables, in which $Y$ is a cluster.

By Theorem \ref{theorem::finitetype}, since $\A(X,T)$ arises from a M\"obius strip, $\X$ contains a finite number of cluster variables. Since the cluster algebra $\A(Y,X^*,T^*)$ is arising from a surface and generates the same finite set $\X$ of cluster variables, $Q$ is one of the following according to Theorem \ref{thm::FominFinite}: a polygon with $n_1$ vertices, a polygon $n_2$ vertices and a puncture, or a M\"obius Strip.

If $\A(Y,X^*,T^*)$ is a M\"obius strip, then the M\"obius strip is unistructural. Thus, we need to prove $\A(Y,X^*,T^*)$ cannot be a polygon or a polygon with a puncture. We proceed by contradiction. If the M\"obius strip is unistructural, the number of cluster variables should be a positive integer. We know the number of cluster variables for a Mobius strip with $m$ marked points is $$\frac{3m^2-m+2}{2},$$ according to Dupont and Palesi \cite{DP15}. The number of cluster variables for a polygon with $n_1$ vertices is the number of almost positive roots of the Type $\mathbb{A}_n$ root system which is given by $\frac{n_1(n_1+3)}{2}$.

Suppose $\A(Y,X^*,T^*)$ is a cluster algebra of a polygon with $n_1$ vertices. To test if the number of cluster variables is an integer between a polygon and the M\"obius Strip, we have 

\begin{align*}
   \frac{3{m}^2-m+2}{2}=\frac{n_1(n_1+3)}{2}.
\end{align*}
The solution of this equation is $n=\frac{1}{2} (-3 \pm \sqrt{17 - 4 m + 12 m^2})$ which proves $\A(X,T)$ and $\A(Y,X^*,T^*)$ cannot have integer number of cluster variables, which is a contradiction.

Similarly, for a polygon with $n_2$ vertices and a puncture, the number of cluster variables is the number of almost positive roots of the Type $\mathbb{D}_n$ root system which is given by $n_2^2$. Then we have, \begin{align*}
   \frac{3{m}^2-m+2}{2}={n_2}^2.
\end{align*}
The solution is $n=\frac{1}{6} (1 \pm \sqrt{12m^2 - 23})$ which cannot be an integer. Therefore, if $\A(Y,X^*,T^*)$ is a cluster algebra of polygon with a puncture, $\A(X,T)$ and $\A(Y,X^*,T^*)$ cannot have integer number of cluster variables, which is also a contradiction.
\end{proof}

\section{Future Directions}

With this definition of a partitioned quiver associated to a non-orientable surface, there are many open questions that we are investigating. The first question we are investigating is defining an appropriate exchange matrix to a quasi-triangulation. In \cite{DP15}, the authors associate an exchange matrix to the double cover $(\mathbf{\Bar{S}, \Bar{M}})$ when the triangulation of $\SM$ contains no quasi-arcs. We hope that using our notion of partitioned quiver, we can define an exchange matrix as a pair $(B_T, P_T)$ where the second entry is a partition of the paths in $T$. With this exchange matrix or using the notion of $T$-paths in \cite{Sch10}, we hope to define $g$-vectors and $F$-polynomials associated to non-orientable surfaces.  

Another project that the first and last author are actively working on with Aaron Chan is the categorification of quasi-triangulations of unpunctured non-orientable surfaces. The additive categorification of cluster algebras arising from orientable surfaces has been studied by many, see for example \cite{Ami11}, \cite{Rei10}, \cite{Kel12}, \cite{ADS14}. We are currently investigating categorification of quasi-triangulations of unpunctured non-orientable surfaces and their quasi-mutations. In particular, we show quasi-triangulations of non-orientable surfaces are in bijection with quivers with quivers with potential equipped with a fixed point free anti-involution. Moreover, we have that arcs and certain closed curves of non-orientable surfaces giver rise to indecomposable representations of the associated Jacobian algebra of the quiver with potential.

Another question all authors hope to work towards is defining a frieze pattern associated to quasi-triangulations of non-orientable surfaces. Now with our definition of the partitioned quiver, we wonder if we can use this data to properly generate frieze patterns.
 \newpage
\bibliographystyle{alpha} 
\bibliography{biblio}

\begin{thebibliography}{BMHL20}

\bibitem[ADS14]{ADS14}
Ibrahim Assem, Gr\'egoire Dupont, and Ralf Schiffler.
\newblock On a category of cluster algebras.
\newblock {\em J. Pure Appl. Algebra}, 218(3):553--582, 2014.

\bibitem[Ami11]{Ami11}
Claire Amiot.
\newblock On generalized cluster categories.
\newblock In {\em Representations of algebras and related topics}, EMS Ser.
  Congr. Rep., pages 1--53. Eur. Math. Soc., Z\"urich, 2011.

\bibitem[BMHL20]{BHL20}
Véronique Bazier-Matte, Ruiyan Huang, and Hanyi Luo.
\newblock Number of triangulations of a möbius strip.
\newblock 9 2020.

\bibitem[CS14]{CS2014}
Leonid Chekhov and Michael Shapiro.
\newblock Teichm{\"u}ller spaces of riemann surfaces with orbifold points of
  arbitrary order and cluster variables.
\newblock {\em International Mathematics Research Notices},
  2014(10):2746--2772, 2014.

\bibitem[DP15]{DP15}
Grégoire Dupont and Frédéric Palesi.
\newblock Quasi-cluster algebras from non-orientable surfaces.
\newblock {\em Journal of Algebraic Combinatorics}, 42(2):429–472, Mar 2015.

\bibitem[FST08]{FST08}
Sergey Fomin, Michael Shapiro, and Dylan Thurston.
\newblock Cluster algebras and triangulated surfaces. part {I}: Cluster
  complexes.
\newblock {\em Acta Mathematica}, 201(1):83--146, 2008.

\bibitem[FZ02]{FZ02}
Sergey Fomin and Andrei Zelevinsky.
\newblock Cluster algebras {I}: Foundations.
\newblock {\em Journal of the American Mathematical Society}, 15(2):497--529,
  2002.

\bibitem[FZ03]{FZ03}
Sergey Fomin and Andrei Zelevinsky.
\newblock Cluster algebras {II}: Finite type classification.
\newblock {\em Inventiones mathematicae}, 154(1):63--121, 2003.

\bibitem[FZ07]{FZ07}
Sergey Fomin and Andrei Zelevinsky.
\newblock Cluster algebras {IV}: Coefficients.
\newblock {\em Compositio Mathematica}, 143(1):112--164, 2007.

\bibitem[Kel12]{Kel12}
Bernhard Keller.
\newblock Cluster algebras and derived categories.
\newblock In {\em Derived categories in algebraic geometry}, EMS Ser. Congr.
  Rep., pages 123--183. Eur. Math. Soc., Z\"urich, 2012.

\bibitem[LP12]{lp}
Thomas Lam and Pavlo Pylyavskyy.
\newblock Laurent phenomenon algebras.
\newblock {\em arXiv preprint arXiv:1206.2611}, 2012.

\bibitem[LP16]{LP16}
Thomas Lam and Pavlo Pylyavskyy.
\newblock Laurent phenomenon algebras.
\newblock {\em Cambridge Journal of Mathematics}, 4(1):121–162, 2016.

\bibitem[MSW13]{MSW13}
Gregg Musiker, Ralf Schiffler, and Lauren Williams.
\newblock Bases for cluster algebras from surfaces.
\newblock {\em Compositio Mathematica}, 149(2):217–263, 2013.

\bibitem[Rei10]{Rei10}
Idun Reiten.
\newblock Cluster categories.
\newblock In {\em Proceedings of the {I}nternational {C}ongress of
  {M}athematicians. {V}olume {I}}, pages 558--594. Hindustan Book Agency, New
  Delhi, 2010.

\bibitem[Sch10]{Sch10}
Ralf Schiffler.
\newblock On cluster algebras arising from unpunctured surfaces {II}.
\newblock {\em Advances in Mathematics}, 223(6):1885 -- 1923, 2010.

\bibitem[Wil15]{wilson2015shellability}
Jon Wilson.
\newblock Shellability and sphericity of the quasi-arc complex of the m\"obius
  strip, 2015.

\bibitem[Wil18]{Wil18}
Jon Wilson.
\newblock Shellability and sphericity of finite quasi-arc complexes.
\newblock {\em Discrete Comput Geom}, 59(3):680--706, 2018.

\bibitem[Wil19]{Wil19}
Jon Wilson.
\newblock Positivity for quasi-cluster algebras, 2019.
\newblock arXiv:1912.12789.

\bibitem[Wil20]{Wil20}
Jon Wilson.
\newblock Laurent phenomenon algebras arising from surfaces {II}: Laminated
  surfaces.
\newblock {\em Selecta Mathematica}, 26, 11 2020.

\end{thebibliography}

\end{document}